\documentclass[preprint,12pt]{article}
\usepackage[top=1in, bottom=1in, left=0.9in, right=0.9in]{geometry}
\usepackage{amsmath,amsfonts,amssymb,amsthm}
\usepackage{hyperref}
\usepackage{mathrsfs,dsfont}
\usepackage{graphicx,color}
\usepackage{colortbl,dcolumn}
\usepackage{psfrag}
\usepackage{booktabs}
\usepackage{cite}
\usepackage{subfigure}
\usepackage{epstopdf}
\usepackage{caption}
\usepackage{url}

\allowdisplaybreaks[4]
\numberwithin{equation}{section}
\usepackage[justification=centering]{caption}

\renewcommand{\footnote}{}
\numberwithin{equation}{section}

\newcommand{\R}{\mathbb{R}}
\newcommand{\N}{\mathbb{N}}

\newcommand{\E}{\mathbb{E}}
\renewcommand{\P}{\mathbb{P}}

\newcommand{\F}{\mathcal{F}}
\newcommand{\diff}{{\,\rm{d}}}
\newcommand{\V}{\mathcal{V}}
\newcommand{\C}{\mathbb{C}}
\newcommand{\e}{\hat{e}}
\newcommand{\D}{\hat{D}}
\newcommand{\bb}{\hat{b}}
\newcommand{\si}{\hat{\sigma}}
\renewcommand{\diff}{{\,\mathrm{d}}}

\newtheorem{theorem}{Theorem}[section]
\newtheorem{definition}[theorem]{Definition}
\newtheorem{lemma}[theorem]{Lemma}

\newtheorem{assumption}[theorem]{Assumption}

\begin{document}
\title{Convergence rate and exponential stability of backward Euler method for neutral stochastic delay differential equations under generalized monotonicity conditions\footnotemark[1]}

\author{Jingjing Cai\footnotemark[2]\qquad
Ziheng Chen\footnotemark[3]\qquad
Yuanling Niu\footnotemark[4]}

\date{}

\maketitle

\footnotetext{\footnotemark[1] This work was supported by National Natural Science Foundation of China (Nos. 12201552, 12031020, 12371417 and 12071488), Yunnan Fundamental Research Projects (No. 202301AU070010) and Innovation Team of School of Mathematics and Statistics of Yunnan University (No. ST20210104).}





\footnotetext{\footnotemark[2] School of Mathematics and Statistics, HNP-LAMA, Central South University, Changsha, 410083, China. Email: 212111072@csu.edu.cn}

\footnotetext{\footnotemark[3] School of Mathematics and Statistics, Yunnan University, Kunming, Yunnan, 650500, China. Email: czh@ynu.edu.cn}

\footnotetext{\footnotemark[4] School of Mathematics and Statistics, HNP-LAMA, Central South University, Changsha, 410083, China. Email: yuanlingniu@csu.edu.cn. Corresponding author.}

\begin{abstract}
      {\rm
      This work focuses on the numerical approximations of neutral stochastic delay differential equations with their drift and diffusion coefficients growing super-linearly with respect to both delay variables and state variables. Under generalized monotonicity conditions, we prove that the backward Euler method not only converges strongly in the mean square sense with order $1/2$, but also inherit the mean square exponential stability of the original equations.
      As a byproduct, we obtain the same results on convergence rate and exponential stability of the backward Euler method for stochastic delay differential equations with generalized monotonicity conditions. These theoretical results are finally supported by several numerical experiments.
      } \\


\end{abstract}

\section{Introduction}

This work concerns the following non-autonomous neutral stochastic delay differential equations (NSDDEs) of It\^{o} form
\begin{equation}\label{eq:NSDDEsdiffform}
      \diff{(X(t)-D(X(t-\tau)))}
	  =
	  b(t,X(t),X(t-\tau))\diff{t}
	  +
	  \sigma(t,X(t),X(t-\tau))\diff{W(t)},
	  \quad t > 0
\end{equation}
with the initial value $X(t) = \phi(t)$ for $t \in [-\tau,0]$, where the delay $\tau > 0$ is a constant. Here, $\{W(t)\}_{t \geq 0}$ is an $m$-dimensional standard Brownian motion on the complete filtered probability space $(\Omega,\F,\P,\{\F_{t}\}_{t \geq 0})$ with the filtration $\{\F_{t}\}_{t \geq 0}$ satisfying the usual conditions. Moreover, $\phi \colon [-\tau,0] \to \R^{d}$, $D \colon \R^{d} \to \R^{d}$, $b \colon [0,+\infty) \times \R^{d} \times \R^{d} \to \R^{d}$ and $\sigma \colon [0,+\infty) \times \R^{d} \times \R^{d} \to \R^{d \times m}$ are Borel measurable functions; see Assumption \ref{ass:mainassone} below for precise conditions. As an important type of stochastic differential equations (SDEs), NSDDEs \eqref{eq:NSDDEsdiffform} not only depend on the past and present states but also involve derivatives with delays and the functions themselves, consequently becoming complex but performing better to capture the complexity and uncertainty of real-world systems; see, e.g., \cite[Chapter 6]{mao2008stochastic}. Nowadays, these equations have been widely used to model different evolution phenomena arising in various fields, such as signal processing in neural networks \cite{hu2015advances}, collision problem in electrodynamics \cite{bao2016asymptotic} and lossless transmission in circuit systems \cite{liao2016dynamical}.

Since the analytic solutions of most NSDDEs are rarely available, numerical approximation has become a powerful tool to study the behaviour of their solutions. Up to now, much progress has been
achieved in constructing and investigating different kinds of numerical methods for NSDDEs with the traditional global Lipschitz condition (see, e.g., \cite{zhang2008meansquare, gan2011meansquare, wang2011meansquare}) and for NSDDEs with the popular global monotonicity condition (see, e.g., \cite{zhou2009convergence, zong2015exponential, ji2017tamed, yan2017strong, yuan2017meansquare, obradovic2019almost, lan2019strong, zhao2020strong, zhang2020convergence, cui2023explicit, gao2024convergence} and references therein). However, both global Lipschitz condition and the global monotonicity condition are still restrictive in the sense that some practical NSDDEs, such as neutral stochastic cellular neural networks \cite{guo2013fixed} and neutral stochastic delay power logistic model \cite{mao2005khasminskii}, violating such conditions. This situation motivates us to establish efficient numerical methods for solving NSDDEs \eqref{eq:NSDDEsdiffform} with generalized monotonicity condition.

Compared with a large amount of valuable results (see, e.g., \cite{hunzenthaler2015numerical, hutzenthaler2020perturbation, wang2023meansquare, wu2022split, chen2021first, yi2021positivity, lei2023strong, dai2023order} and references therein) for the usual SDEs beyond global monotonicity condition, just a very limited number of literature \cite{deng2021tamed, mao2011numerical, guo2018truncated, fei2020advances, song2022strong} focus on the numerical approximation of NSDDEs with generalized monotonicity condition. The authors in \cite{deng2021tamed} propose two types of explicit tamed Euler--Maruyama (EM) schemes for NSDDEs and derive the $L^{p}(\Omega;\R^{d})$-norm convergence for the first type scheme, the mean square convergence rate for the second type scheme and the mean square exponential stability for the partially tamed EM scheme. It is worth emphasizing that the remainder works \cite{mao2011numerical, guo2018truncated, fei2020advances, song2022strong} actually deal with numerical approximations of stochastic delay differential equations (SDDEs), i.e., NSDDEs \eqref{eq:NSDDEsdiffform} with the neutral term $D \equiv 0$, under the generalized Khasminskii-type condition. The seminal work \cite{mao2011numerical} proves that the explicit EM numerical solutions converge to the true solutions in probability. Subsequently, the strong convergence (without rate) \cite{guo2018truncated}, the almost optimal strong convergence rate \cite{fei2020advances} and the optimal strong convergence rate \cite{song2022strong} of the explicit truncated EM method have been discussed in sequence. Although the explicit numerical methods possess the advantage of simple algebraic structures and low computational costs, they usually face a severe stepsize restriction in solving stiff SDEs, which is due to stability issues and makes the overall computational costs extremely expensive; see, e.g., \cite{chen2020convergence, wang2020meansquare, milstein2021stochastic, wang2023meansquare} and Section \ref{section5}. Owing to these facts, our aim here is to establish the strong convergence rate and exponential stability in the mean square sense of an implicit numerical method, i.e., the backward Euler method \eqref{eq:BEmethodintro}, for \eqref{eq:NSDDEsdiffform} under generalized monotonicity condition.

As a class of standard implicit numerical method, the backward Euler method and its variant methods have already been applied to SDEs \cite{higham2002strong, mao2013strong, hunzenthaler2015numerical, beyn2016stochastic, andersson2017meansquare, wang2020meansquare, liu2022lpconvergence}, SDDEs \cite{mao2003numerical, zong2015theta, zhou2017strong, zhang2019backward, yue2021strong} and NSDDEs \cite{yin2011convergence, yan2017strong, obradovic2019implicit} with non-globally Lipschitz conditions. All these existing results indicate that the backward Euler method is able to provide strongly convergent numerical approximations for the considered equations satisfying polynomial growth condition and global monotonicity condition. This theory will be further strengthened by one of our main results Theorem \ref{th:mul-noise-rate}, showing that numerical solutions generated by the backward Euler method \eqref{eq:BEmethodintro} converges strongly with order $1/2$ in the mean square sense to exact solutions of NSDDEs \eqref{eq:NSDDEsdiffform} under polynomial growth condition \eqref{eq:bpolynomialgrow} and generalized monotonicity condition \eqref{eq:generalmonocond}. Concerning another main result of this work, we also prove that the backward Euler method \eqref{eq:BEmethodintro} is able to inherit the mean square exponential stability of NSDDEs \eqref{eq:NSDDEsdiffform} under another
generalized monotonicity condition \eqref{eq:genmonostability}, which is weaker than Assumptions 5.1 and 5.2 in \cite{deng2021tamed}. As a direct consequence of our main results, we additionally derive the convergence rate and exponentially stability in the mean square sense of the backward Euler method for SDDEs with global monotonicity condition (i.e., $V \equiv 0$ in \eqref{eq:generalmonocond} and \eqref{eq:genmonostability}) or generalized monotonicity condition (i.e., $V \neq 0$ in \eqref{eq:generalmonocond} and \eqref{eq:genmonostability}). Both cases allow the drift and diffusion coefficients of SDDEs to super-linearly grow not only with respect to the delay variables but also to the state variables, which has not been reported before (see, e.g., \cite{mao2003numerical, zong2015theta, zhou2017strong, cao2021strong, zhang2019backward, yue2021strong} for more details).

The remainder of this paper is organized as follows. The next section provides the upper mean square error bound for the backward Euler method, which will be exploited to establish the corresponding mean square convergence rate in Section \ref{section3}. Section \ref{section4} discusses the mean square exponential stability for NSDDEs and the considered method. Finally, some numerical experiments are conducted to verify the theoretical results in Section \ref{section5}.

\section{Preliminaries and upper mean square error bound}\label{section2}

We start with some notation that will be used throughout the paper. Let $\R^{+} := [0,+\infty), \R := (-\infty,+\infty)$ and $\N := \{1,2,\cdots\}$ be the set of all positive integers. For any $a,b \in \R$, we set $\lceil a\rceil$ as the smallest integer that is greater than or equal to $a$, $a \vee b := \max\{a,b\}$ and $a \wedge b := \min\{a,b\}$. Let $\langle \cdot, \cdot \rangle$ and $|\cdot|$ and denote the Euclidean inner product and the corresponding norm of vectors in $\R^{d}$. By $A^{\top}$, we denote the transpose of a vector $A \in \R^{d}$ or a matrix $A \in \R^{d \times m}$. For any matrix $A$, we use $|A| := \sqrt{\mathrm{trace}(A^{\top}A)}$ to denote its trace norm. For any $p \geq 1$, let $(L^{p}(\Omega;\R^{d}), |\cdot|_{L^{p}(\Omega;\R^{d})})$ be the Banach space of all $\R^{d}$-valued random variables $\eta$ on the probability space $(\Omega,\F,\P)$ with norm $|\eta|_{L^{p}(\Omega;\R^{d})} := (\E[|\eta|^{p}])^{\frac{1}{p}} < \infty$. Fix $\tau > 0$, we denote by $\mathbb{C}([-\tau,0];\R^{d})$ the space of all continuous functions $\phi \colon [-\tau,0] \to \R^{d}$ equipped with the norm $|\phi|_{\mathbb{C}([-\tau,0];\R^{d})} := \sup_{t \in [-\tau,0]}|\phi(t)|$. By $\mathbb{C}( \mathbb{R}^d;\mathbb{R}^{+} )$, we denote the space of all nonnegative continuous functions defined on $\R^d$. By $\mathcal{V}(\R^{d} \times \R^{d};\R^{+})$, we denote the space of all functions $V \colon \R^{d} \times \R^{d}$ satisfying $V(x,x) = 0$ for all $x \in \R^{d}$. For simplicity, the letter $C$ will be used to stand for a generic positive constant whose value may vary for each appearance, but
independent of the stepsize of the considered numerical method.

In order to ensure the well-posedness of NSDDEs \eqref{eq:NSDDEsdiffform}, we give the following conditions.

\begin{assumption}\label{ass:mainassone}
      Let $D(0) = 0$ and assume that there exists $\nu \in (0,1)$ such that
      \begin{align}\label{eq:DCompression}
			|D(x) - D(y)|
            \leq \nu |x - y|,
            \quad x,y \in \R^d.
      \end{align}
      Let $| b(t,0,0) | + | \sigma(t,0,0) | \leq K$ for all $t \geq 0$ and some constant $K > 0$. Also let $b$ and $\sigma$ satisfy the polynomial growth condition, i.e., there exist $q > 1$ and $K > 0$ such that for all $t \geq 0$ and $x_{1}, x_{2}, y_{1}, y_{2} \in \R^{d}$,
      \begin{align}\label{eq:bpolynomialgrow}
            &~|b(t,x_{1},y_{1}) - b(t,x_{2},y_{2})|
            +
            |\sigma(t,x_{1},y_{1}) - \sigma(t,x_{2},y_{2})|
            \notag
            \\\leq&~
            K(1 + |x_{1}|^{q-1} + |x_{2}|^{q-1}
            + |y_{1}|^{q-1} + |y_{2}|^{q-1})
            (|x_{1}-x_{2}| + |y_{1}-y_{2}|).
      \end{align}
      Further, let $b$ and $\sigma$ satisfy the Khasminskii-type condition, i.e., there exist $p^{*} \geq 4q-2$, $K_{1},K_{2}\geq 0$ and function $V_1\in\C(\R^{d};\R^{+})$ such that
      \begin{align}\label{eq:globalcoercivitycond}
            &~(1+|x|^{2})^{\frac{p^{*}}{2}-1}
            \big(2\langle x-D(y),b(t,x,y) \rangle
			+
			(p^*-1)|\sigma(t,x,y)|^{2}\big)\notag
			\\\leq&~
			K_{1}(1+|x|^{p^{*}}+|y|^{p^{*}})-K_{2}(V_1(x)-V_1(y)),
			\quad t \geq 0, x,y \in \R^{d}.
      \end{align}
\end{assumption}

Owing to $|b(t,0,0)| + |\sigma(t,0,0)| \leq K$ for all $t \geq 0$ and \eqref{eq:bpolynomialgrow}, one can present the super-linearly growing bound for $b$ and $\sigma$, i.e., there exists $K > 0$ such that
\begin{align}\label{eq:bsuperlinear}
      |b(t,x,y)|+|\sigma(t,x,y)|
      \leq
      K(1 +|x|^{q} + |y|^{q}),
      \quad t \geq 0, x,y \in \R^{d},
\end{align}
which indicates that the coefficients $b$ and $\sigma$ grow super-linearly not only with respect to the delay variables but also to the state variables. Furthermore, \eqref{eq:bpolynomialgrow} shows that $b$ and $\sigma$ are locally Lipschitz continuous, i.e., for every $R > 0$, there exists $K(R) > 0$, depending on $R$, such that for all $t \geq 0$ and $x_{1},x_{2},y_{1},y_{2} \in \R^{d}$ with $|x_{1}| \vee |x_{2}| \vee |y_{1}| \vee |y_{2}| \leq R$,
\begin{align*}
      |b(t,x_{1},y_{1})-b(t,x_{2},y_{2})|
      \vee
      |\sigma(t,x_{1},y_{1})-\sigma(t,x_{2},y_{2})|
      \leq
      K(R)(|x_{1}-x_{2}| + |y_{1}-y_{2}|).
\end{align*}
Together with \eqref{eq:DCompression} and \eqref{eq:globalcoercivitycond}, Theorem 3.1 in \cite{luo2006newcriteria} guarantees that \eqref{eq:NSDDEsdiffform} admits a unique global solution $\{X(t)\}_{t \geq -\tau}$, described by $X(t) = \phi(t)$ for all $t \in [-\tau,0]$ and
\begin{align}\label{eq:NSDDEsintegral}
      X(t)-D(X(t-\tau))
      =&~
      \phi(0)-D(\phi(-\tau))
      +
      \int_{0}^{t} b(s,X(s),X(s-\tau)) \diff{s}
      \notag
      \\&~+
      \int_{0}^{t} \sigma(s,X(s),X(s-\tau))\diff{W(s)},
      \quad t \geq 0.
\end{align}
Following the proof of \cite[Theorem 4.5 in Chapter 6]{mao2008stochastic}, one can further bound the moment of $\{X(t)\}_{t \in [-\tau,T]}$ for any given $T > 0$, i.e., there exists $K > 0$, depending on $\phi, p^{*}$ and $T$, such that
\begin{align}\label{eq:Xmomentbound}
      \sup_{-\tau \leq t \leq T} \E\big[|X(t)|^{p^{*}}\big]
      \leq
      K.
\end{align}

Armed with these preparations, we are in a position to consider the numerical approximations of \eqref{eq:NSDDEsdiffform}. To this end, let $\Delta > 0$ be an equidistant stepsize with $\tau = M\Delta$ for some $M \in \N$ and let $t_{n} := n\Delta, n = -M,-M+1,\cdots$ be the corresponding uniform grids. The backward Euler method applied to \eqref{eq:NSDDEsdiffform} is given by
$X_{n} = \phi(t_{n})$ for $n = -M, -M+1, \cdots, 0$ and
\begin{align}\label{eq:BEmethodintro}
      X_{n+1}
      =&~
      D(X_{n+1-M}) - D(X_{n-M}) + X_{n}
      +
      b(t_{n+1},X_{n+1},X_{n+1-M})\Delta
      \notag\\&~+
      \sigma(t_{n},X_{n},X_{n-M})\Delta{W_{n}},
      \quad n = 0, 1,\cdots,
\end{align}
where $X_{n}$ denotes the numerical approximation of $X(t_{n})$ and $\Delta{W_{n}}:=W(t_{n+1})-W(t_{n})$ is the Brownian increment. Concerning the well-posedness and convergence analysis of \eqref{eq:BEmethodintro}, we put the following generalized monotone condition on the coefficients of \eqref{eq:NSDDEsdiffform}.

\begin{assumption}\label{ass:mainasstwo}
      Assume that there exist constants $\eta > 1$, $L > 0$ and function  $V \in \V(\R^d \times \R^d;\R^{+})$ such that
      \begin{align}\label{eq:generalmonocond}
            &~2\langle(x_{1}-x_{2})-(D(y_{1})-D(y_{2})),
			b(t,x_{1},y_{1})-b(t,x_{2},y_{2}) \rangle
			+
            \eta|\sigma(t,x_{1},y_{1})
			-\sigma(t,x_{2},y_{2})|^{2}
            \notag
			\\\leq&~
            L(|x_{1}-x_{2}|^{2}
            + |y_{1}-y_{2}|^{2})
            - V(x_1,x_2)+V(y_1,y_2),
            \quad
            t \geq 0, x_{1},x_{2},y_{1},y_{2} \in \R^{d}.
		\end{align}
\end{assumption}

Under Assumption \ref{ass:mainasstwo}, we establish the well-posedness of \eqref{eq:BEmethodintro}.

\begin{lemma}\label{well-posedness-num}
      Suppose that Assumption \ref{ass:mainasstwo} holds with $L\Delta <2$. Then \eqref{eq:BEmethodintro} admits a unique solution $\{X_{n}\}_{n \geq -M}$ with probability one.
\end{lemma}

\begin{proof}
      For each step of the backward Euler method \eqref{eq:BEmethodintro} with $X_{n-1}$ given, to find $X_{n}$ via \eqref{eq:BEmethodintro} is equivalent to solve the implicit problem $H(x) := x - b(t_{n}, x, a) \Delta + c = 0$, where $a, c \in \R^d$ are known numbers. Utilizing \eqref{eq:generalmonocond} and $V \in \V(\R^d \times \R^d;\R^{+})$ shows that for any $x_{1}, x_{2} \in \R^{d}$,
      \begin{align*}
            \langle x_1-x_2 , H(x_1)-H(x_2)\rangle
            =&~
            |x_1 - x_2|^2
            -
            \Delta \langle x_1-x_2,b(t_{n}, x_1,a)-b(t_{n}, x_2,a)\rangle
            \\\geq&~
            |x_1 - x_2|^2
            -
            \frac{\Delta}{2}(L|x_1 - x_2|^{2}-V(x_1,x_2))
            \\\geq&~
            \Big(1-\frac{L}{2}\Delta\Big)|x_1 - x_2|^{2}.
      \end{align*}
      Then we complete the proof via the uniform monotonicity theorem \cite[Theorem C.2]{stuart1996dynamical}.
\end{proof}

In order to perform the convergence analysis of \eqref{eq:BEmethodintro} on any finite time interval, we fix $T = N\Delta > 0$ for some $N \in \N$ and define the local truncated error. More precisely, we first replace $X_{n}$ in \eqref{eq:BEmethodintro} by $X(t_{n})$ for $n = -M, -M+1, \cdots,N$ to get
\begin{align*}
      \tilde{X}_{n+1}
      :=&~
      D(X(t_{n+1-M})) - D(X(t_{n-M})) + X(t_{n})
      \\&~+
      b(t_{n+1},X(t_{n+1}),X(t_{n+1-M}))\Delta
      +
      \sigma(t_{n},X(t_{n}),X(t_{n-M}))\Delta{W_{n}},
\end{align*}
and then define $R_{n+1} := X(t_{n+1})-\tilde{X}_{n+1}$ as the local truncated error, i.e.,
\begin{align}\label{eq:localerror}
      R_{n+1}
      =&~
      X(t_{n+1})-X(t_{n})
      -
      (D(X(t_{n+1-M})) - D(X(t_{n-M})))
      \notag
      \\&~-
      b(t_{n+1},X(t_{n+1}),X(t_{n+1-M})) \Delta
      -
      \sigma(t_{n},X(t_{n}),X(t_{n-M}))\Delta W_{n}.
\end{align}
As a consequence of \eqref{eq:BEmethodintro} and \eqref{eq:localerror}, we have
\begin{align}\label{eq:error}
      &~\big(X(t_{n}) - X_{n}\big) - \big(D(X(t_{n-M})) - D(X_{n-M})\big)
      \notag
      \\=&~
      \big(X(t_{n-1}) - X_{n-1}\big)
      - \big(D(X(t_{n-1-M}) - D(X_{n-1-M})\big)
      \notag
      \\&~+
      \big(b(t_{n},X(t_{n}),X(t_{n-M}))
      - b(t_{n},X_{n},X_{n-M})\big)\Delta
      \notag
      \\&~+
      \big(\sigma(t_{n-1},X(t_{n-1}),X(t_{n-1-M}))
      -\sigma(t_{n-1},X_{n-1},X_{n-1-M})\big)\Delta{W_{n-1}} + R_{n}.
\end{align}
For any $n = 1,2,\cdots, N$, setting $e_{n}:= X(t_{n}) - X_{n}$, $\Delta D_{n} := D(X(t_{n-M})) - D(X_{n-M})$ and
\begin{align*}
      &~
      \Delta b_{n}
      :=
      b(t_{n},X(t_{n}),X(t_{n-M}))-b(t_{n},X_{n},X_{n-M}),
      \\&~
      \Delta \sigma_{n}
      :=
      \sigma(t_{n},X(t_{n}),X(t_{n-M}))-\sigma(t_{n},X_{n},X_{n-M})
\end{align*}
enables us to rewrite \eqref{eq:error} as
\begin{equation}\label{eq:simplerror}
      e_{n} - \Delta D_{n}
      =
      e_{n-1} - \Delta D_{n-1} + \Delta b_{n} \Delta
      +
      \Delta \sigma_{n-1}\Delta{W_{n-1}} + R_{n}.
\end{equation}

The following lemma provides a upper mean square error bounds via the previous local truncated error, which is a key tool to establish the mean square convergence rate of \eqref{eq:BEmethodintro}.

\begin{lemma}\label{upper error bounds}
      Suppose that Assumptions \ref{ass:mainassone} and \ref{ass:mainasstwo} hold with $\nu + L\Delta \leq \rho$ for some  $\rho \in (\nu,1)$. Then there exists $C > 0$, independent of $\Delta$, such that for all $n = 1,2,\cdots,N$,
      \begin{equation}\label{eq:upper error bounds}
            \E\big[|X(t_{n}) - X_{n}|^{2}\big]
			\leq
            C\bigg(\sum_{i=1}^{N} \E\big[|R_i|^2\big]
            +
            \Delta^{-1}\sum_{i=1}^{N}
            \E\big[|\E({R_i}~|~\F_{t_{i-1}})|^{2}\big]\bigg).
      \end{equation}
\end{lemma}

\begin{proof}
      Based on the identity
      \begin{align}\label{eq:identity}
            |u|^{2}-|v|^{2} + |u-v|^{2}
            =
            2\langle u,u-v \rangle,
            \quad u,v \in \R^{d},
      \end{align}
      we use \eqref{eq:simplerror} to obtain that for all $n = 1,2,\cdots,N$,
      \begin{align}\label{eq:identity-en-1}
            &~|e_{n} - \Delta D_{n}|^{2}
            -
            |e_{n-1} - \Delta D_{n-1}|^{2}
            +
            |(e_{n}-\Delta D_{n}) - (e_{n-1}-\Delta D_{n-1})|^{2}
            \notag
            \\=&~
            2\langle e_n-\Delta D_{n},
            (e_{n}-\Delta D_{n})-(e_{n-1}-\Delta D_{n-1})\rangle
            \notag
            \\=&~
            2\langle e_n-\Delta D_{n},
            \Delta b_{n} \Delta +
            \Delta \sigma_{n-1}\Delta{W_{n-1}} + R_{n}\rangle
            \notag
            \\=&~
            2\langle e_n-\Delta D_{n},
            \Delta b_{n} \Delta \rangle
            +
            2\langle (e_n-\Delta D_{n})-(e_{n-1}-\Delta D_{n-1}),
            \Delta \sigma_{n-1}\Delta{W_{n-1}} + R_{n}\rangle
            \notag
            \\&~+
            2\langle e_{n-1}-\Delta D_{n-1},
            \Delta \sigma_{n-1}\Delta{W_{n-1}}\rangle
            +
            2\langle e_{n-1}-\Delta D_{n-1},
            R_{n}\rangle.
      \end{align}
      By the properties of conditional expectations and the independence between $\Delta W_{n-1}$ and $\F_{t_{n-1}}$, one has
      \begin{align*}
            &~\E\big[\langle e_{n-1}-\Delta D_{n-1},
            \Delta\sigma_{n-1}\Delta W_{n-1} \rangle \big]
            \\=&~
            \E\big[\E\big(\langle e_{n-1}-\Delta D_{n-1},
            \Delta\sigma_{n-1}\Delta W_{n-1} \rangle ~|~ \F_{t_{n-1}} \big) \big]
            \\=&~					
            \E\big[\langle e_{n-1}-\Delta D_{n-1},
            \Delta \sigma_{n-1}
            \E(\Delta W_{n-1} ~|~ \F_{t_{n-1}})
            \rangle\big]
			\\=&~
			0.
      \end{align*}
      Taking expectations on the both sides of \eqref{eq:identity-en-1} leads to
      \begin{align}\label{eq:identity-E-en}
            &~\E\big[|e_{n} - \Delta D_{n}|^{2}\big]
            -
            \E\big[|e_{n-1} - \Delta D_{n-1}|^{2}\big]
            +
            \E\big[ |(e_{n}-\Delta D_{n})
            - (e_{n-1}-\Delta D_{n-1})|^{2}\big]
            \notag
            \\=&~
            2\Delta \E\big[\langle
            e_n-\Delta D_{n},\Delta b_{n} \rangle \big]
            +
            2\E\big[\langle e_{n-1}-\Delta D_{n-1},R_{n}\rangle \big]
            \notag
            \\&~+
            2\E\big[\langle (e_{n}-\Delta D_{n})-(e_{n-1}-\Delta D_{n-1}),
            \Delta \sigma_{n-1}\Delta W_{n-1} + R_{n}\rangle \big]
            \notag
            \\=:&~
            I_1 + I_2 + I_3.
      \end{align}
      We will estimate $I_1,I_2$ and $I_3$ one by one. Applying \eqref{eq:generalmonocond} yields
      \begin{align}\label{eq:I_1}
            I_1
            \leq&~
            L\Delta\big(\E[|e_{n}|^{2}]
            +
            \E[|e_{n-M}|^{2}]\big)
            -
            \eta\Delta \E[|\Delta\sigma_{n}|^2]
            \notag
            \\&~+
            \Delta \E\big[-V(X(t_{n}),X_{n})
            + V(X(t_{n-M}),X_{n-M})\big].
      \end{align}
      For $I_{2}$, we utilize the properties of conditional expectation, the Cauchy--Schwarz inequality and the weighted Young inequality $ 2ab\leq \varepsilon a^2+\frac{b^2}{\varepsilon}$ for all $a,b\in \R$ with $\varepsilon = \Delta > 0$ to get
      \begin{align}\label{eq:I_2}
            I_2
            =&~
            2 \E \big[\E\big(\langle
            e_{n-1}-\Delta D_{n-1}, R_{n} \rangle
            ~|~ \F_{t_{n-1}} \big)\big]
            \notag
            \\=&~
            2 \E \big[\langle e_{n-1}-\Delta D_{n-1},
            \E(R_{n}~|~\F_{t_{n-1}})\rangle\big]
            \notag
            \\\leq&~
            \Delta\E\big[|e_{n-1}-\Delta D_{n-1}|^{2}\big]
            +
            \Delta^{-1}\E\big[|\E(R_{n}~|~\F_{t_{n-1}})|^{2}\big].
      \end{align}
      Concerning $I_{3}$, the weighted Young inequality $2ab \leq \varepsilon a^2+\frac{b^2}{\varepsilon}$ for all $a,b \in \R$ with $\varepsilon = \eta-1 > 0$ and the fact $\E\big[|\Delta\sigma_{n-1}\Delta W_{n-1}|^2\big] = \Delta\E\big[|\Delta\sigma_{n-1}|^2\big]$ help us to derive
      \begin{align}\label{eq:I_3}
            I_{3}
            \leq&~
            \E\big[|(e_{n}-\Delta D_{n})
            - (e_{n-1}-\Delta D_{n-1})|^2\big]
            +
            \E\big[|\Delta\sigma_{n-1}\Delta W_{n-1} + R_{n}|^2\big]
            \notag
            \\=&~
            \E\big[|(e_{n}-\Delta D_{n})
            - (e_{n-1}-\Delta D_{n-1})|^2\big]
            +
            \E\big[|\Delta \sigma_{n-1} \Delta W_{n-1}|^2\big]
            \notag
            \\&~+
            2\E\big[\langle \Delta \sigma_{n-1} \Delta W_{n-1},
            R_{n}\rangle\big]
            +
            \E\big[|R_{n}|^2\big]
            \notag
			\\\leq&~
            \E\big[|(e_{n}-\Delta D_{n})
            - (e_{n-1}-\Delta D_{n-1})|^{2}\big]
            +
            \eta\Delta\E\big[|\Delta\sigma_{n-1}|^{2}\big]
            +
            \frac{\eta}{\eta-1}\E\big[|R_{n}|^{2}\big].
      \end{align}
      Inserting \eqref{eq:I_1}, \eqref{eq:I_2} and \eqref{eq:I_3} into \eqref{eq:identity-E-en} shows
      \begin{align}
            &~\E \big[|e_{n} - \Delta D_{n}|^2\big]
            -
            \E \big[|e_{n-1} - \Delta D_{n-1}|^2\big]
            \notag
            \\\leq&~
            L\Delta\big(\E[|e_{n}|^2] + \E[|e_{n-M}|^{2}]\big)
            +
            \eta\Delta\big(\E[|\Delta\sigma_{n-1}|^2]
            -\E[|\Delta\sigma_{n}|^{2}]\big)
            \notag
            \\&~+
            \Delta \E\big[-V(X(t_{n}),X_{n})
            + V(X(t_{n-M}),X_{n-M})\big]
            +
            \frac{\eta}{\eta-1}\E\big[|R_{n}|^{2}\big]
            \notag
			\\&~+
            \Delta^{-1}\E\big[|\E(R_{n}~|~\F_{t_{n-1}})|^{2}\big]
            +
            \Delta\E\big[|e_{n-1}-\Delta D_{n-1}|^{2}\big].
            \notag
      \end{align}
      By summation, we have
      \begin{align}\label{eq:219219}
            &~\E\big[|e_{n} - \Delta D_{n}|^{2}\big]
            - \E\big[|e_{0} - \Delta D_{0}|^{2}\big]
            \notag
            \\\leq&~
            L\Delta \sum_{i=1}^{n}
            \big(\E[|e_{i}|^{2}] + \E[|e_{i-M}|^{2}]\big)
            +
			\eta\Delta \big(\E[|\Delta\sigma_{0}|^2]
			-
			\E[|\Delta\sigma_{n}|^2]\big)
            \notag
            \\&~+
            \Delta \sum_{i=1}^{n}
            \E\big[-V(X(t_{i}),X_{i})
            + V(X(t_{i-M}),X_{i-M})\big]
            +
            \frac{\eta}{\eta-1}\sum_{i=1}^{n}
            \E\big[|R_{i}|^2\big]
            \notag
            \\&~+
            \Delta^{-1} \sum_{i=1}^{n}
            \E\big[|\E(R_{i}~|~\F_{t_{i-1}})|^{2}\big]
            +
            \Delta \sum_{i=0}^{n-1}
            \E\big[|e_{i}-\Delta D_{i}|^{2}\big].
      \end{align}
      Here we claim that for any $n = 1,2, \cdots,N$,
      \begin{equation}\label{eq:hatV}
            \sum_{i=1}^{n}\E\big[-V(X(t_i),X_i)
            + V(X(t_{i-M}),X_{i-M})\big]
			\leq
			0.
      \end{equation}
      Actually, for the case $n \leq M$, we use $V \in \V(\R^d \times \R^d;\R^{+})$ and $V(X(t_{i-M}),X_{i-M}) = 0$ for $i = 1,2 ,\cdots, n$ to get
      \begin{equation*}\label{hatV1}
            \sum_{i=1}^{n}\E\big[ -V(X(t_i),X_i)
            + V(X(t_{i-M}),X_{i-M}) \big ]
            =
            - \sum_{i=1}^{n}\E\big[V(X(t_i),X_i)\big]
            \leq
            0.
      \end{equation*}		
      For the case $n > M$, the fact $V \in \V(\R^d \times \R^d;\R^{+})$ further shows
      \begin{equation*}\label{hatV2}
            \sum_{i=1}^{n}\E\big[-V(X(t_i),X_i)
            + V(X(t_{i-M}),X_{i-M})\big]
            =
            -\sum_{i=n-M+1}^{n}\E\big[V(X(t_i),X_i)\big]
			\leq
			0.
      \end{equation*}
      Noting that \eqref{eq:DCompression} gives
      \begin{align*}
            \E\big[|e_{i} - \Delta D_{i}|^2\big]
            \leq
            2\E\big[|e_{i}|^2\big]
            +
            2\E\big[|\Delta D_{i}|^2\big]
			\leq
            2\E\big[|e_{i}|^2\big]
            +
            2\nu^2\E\big[|e_{i-M}|^2\big],
      \end{align*}
      we take advantage of \eqref{eq:219219}, \eqref{eq:hatV} and $X(t_i) = X_i = \phi(t_{i}),i= -M,-M+1, \cdots, 0$ to derive
      \begin{align}\label{eq:interation1}
            \E \big[|e_{n}-\Delta D_{n}|^{2}\big]
            \leq&~
            L\Delta\sum_{i=1}^{n}
            \big(\E[|e_{i}|^2] + \E[|e_{i-M}|^2]\big)
            +
            2\Delta \sum_{i=0}^{n-1}
            \E\big[|e_{i}|^{2}\big]
            +
            2\nu^2\Delta \sum_{i=0}^{n-1}
            \E\big[|e_{i-M}|^{2}\big]
            \notag
            \\&~+
            \frac{\eta}{\eta-1}\sum_{i=1}^{N}
            \E\big[|R_{i}|^{2}\big]
			+
            \Delta^{-1} \sum_{i=1}^{N}
            \E\big[|\E(R_{i}~|~\F_{t_{i-1}})|^{2}\big],
      \end{align}
      and consequently
      \begin{align}\label{eq:interation}
            \E \big[|e_{n}-\Delta D_{n}|^{2}\big]
            \leq&~
            L\Delta \E[|e_{n}|^{2}]
            +
            (2L+2+2\nu^2) \Delta
            \sum_{i=0}^{n-1} \E[|e_{i}|^2]
            \notag
            \\&~+
            \frac{\eta}{\eta-1} \sum_{i=1}^{N}
            \E\big[|R_{i}|^{2}\big]
            +
            \Delta^{-1} \sum_{i=1}^{N}
            \E\big[|\E(R_{i}~|~\F_{t_{i-1}})|^2\big].
      \end{align}
      Applying \eqref{eq:DCompression} and the weighted Young inequality $ 2ab\leq \varepsilon a^2 + \frac{b^2}{\varepsilon} $ for all $a,b \in \R$ with $\varepsilon = \frac{1-\nu}{\nu} > 0$ yields
      \begin{align}\label{eq:e_young}
            \E\big[|e_{n}|^{2}\big]
            =&~
            \E\big[|e_{n}-\Delta D_{n}|^{2}\big]
            +
            2\E\big[\langle e_{n}-\Delta D_{n},
            \Delta D_{n} \rangle \big]
            +
            \E\big[|\Delta D_{n}|^{2}\big]
            \notag
            \\\leq&~
			(1+\varepsilon^{-1})
            \E\big[|e_{n}-\Delta D_{n}|^2\big]
            +
            (1+\varepsilon)
            \E\big[|\Delta D_{n}|^{2}\big]
            \notag
			\\\leq&~
            \frac{1}{1-\nu}
            \E\big[|e_{n}-\Delta D_{n}|^2\big]
            +
			\nu\E\big[|e_{n-M}|^{2}\big].
      \end{align}
      Observing that for any $n = 1,2,\cdots,N$, there exists $r \in \{1,2,\cdots, \lceil\frac{T}{\tau}\rceil\}$ such that $n-rM \leq 0$ and $n-(r-1)M > 0$,
      we use \eqref{eq:e_young} and $e_{n-rM} = 0$ to show
      \begin{align}\label{eq:e-interation}
            \E\big[|e_{n}|^2\big]
            \leq&~
			\frac{1}{1-\nu}
            \E\big[|e_{n}-\Delta D_{n}|^{2}\big]
            +
            \frac{\nu}{1-\nu}
            \E\big[|e_{n-M}-\Delta D_{n-M}|^2\big]
            +
            \nu^2 \E\big[|e_{n-2M}|^{2}\big]
            \notag
            \\\leq&~\cdots
            \notag
            \\\leq&~
			\frac{1}{1-\nu}
            \E\big[|e_{n}-\Delta D_{n}|^{2}\big]
            +
            \nu^{r} \E\big[|e_{n-rM}|^{2}\big]
            +
            \frac{1}{1-\nu}
            \Big(\nu\E\big[|e_{n-M}-\Delta D_{n-M}|^2\big]
            \notag
            \\&~+
            \nu^{2}\E\big[|e_{n-2M}-\Delta D_{n-2M}|^2\big]
            +
            \cdots
            +
            \nu^{r-1}\E\big[|e_{n-(r-1)M}
            - \Delta D_{n-(r-1)M}|^2\big]\Big)
            \notag
            \\=&~
            \frac{1}{1-\nu}
            \E\big[|e_{n}-\Delta D_{n}|^2\big]
            +
            \frac{1}{1-\nu}\sum_{j=1}^{r-1}\nu^{j}
            \E\big[|e_{n-jM} - \Delta D_{n-jM}|^2\big].	
	  \end{align}
      As \eqref{eq:interation1} indicates that for any $j = 1, 2,\cdots,r-1$,
      \begin{align*}
            \E\big[|e_{n-jM}-\Delta D_{n-jM}|^{2}\big]
            \leq&~
            (2L+2+2\nu^2) \Delta
            \sum_{i=0}^{n-1} \E[|e_{i}|^2]
            \notag
            +
            \frac{\eta}{\eta-1} \sum_{i=1}^{N}
            \E\big[|R_{i}|^{2}\big]
            \notag
            \\&~+
            \Delta^{-1} \sum_{i=1}^{N}
            \E\big[|\E(R_{i}~|~\F_{t_{i-1}})|^2\big],
      \end{align*}
      we use \eqref{eq:interation} and \eqref{eq:e-interation} to see that
      \begin{align}
            \E \big[|e_{n}|^2 \big]
			\leq&~
            \frac{L\Delta}{1-\nu}\E\big[|e_{n}|^2\big]
            +
			\bigg(\sum_{j=0}^{r-1}\nu^{j}\bigg)
            \frac{2L+2+2\nu^2}{1-\nu}
            \Delta\sum_{i=0}^{n-1}\E[|e_{i}|^2]
            \notag
			\\&~+
            \frac{1}{1-\nu}
            \bigg(\sum_{j=0}^{r-1}\nu^{j}\bigg)
			\bigg(\frac{\eta}{\eta-1}
            \sum_{i=1}^{N} \E\big[|R_{i}|^{2}\big]
            +
			\Delta^{-1} \sum_{i=1}^{N}
            \E\big[|\E(R_{i}~|~\F_{t_{i-1}})|^2\big]\bigg).
            \notag
      \end{align}
      Together with $\nu + L\Delta \leq \rho < 1$, we get
      \begin{align}
            \E \big[|e_{n}|^2 \big]
			\leq&~
			\bigg(\sum_{j=0}^{\lceil\frac{T}{\tau}\rceil-1}\nu^{j}\bigg)
            \frac{2L+2+2\nu^{2}}{1-\rho}
            \Delta\sum_{i=0}^{n-1}\E[|e_{i}|^2]
            \notag
			+
            \frac{\eta}{(1-\rho)(\eta-1)}
            \\&~\times\bigg(\sum_{j=0}^{\lceil\frac{T}{\tau}\rceil-1}\nu^{j}\bigg)
			\bigg(
            \sum_{i=1}^{N} \E\big[|R_{i}|^{2}\big]
            +
			\Delta^{-1} \sum_{i=1}^{N}
            \E\big[|\E(R_{i}~|~\F_{t_{i-1}})|^2\big]\bigg).
            \notag
      \end{align}
      By the discrete Gronwall inequality
      \cite[Lemma 3.4]{mao2013strong}, we obtain \eqref{eq:upper error bounds} and end the proof.
\end{proof}

\section{Mean square convergence rate}\label{section3}

This section aims to develop the mean square convergence rate of \eqref{eq:BEmethodintro} on any finite time $[-\tau, T]$. For this purpose, the H\"{o}lder continuity of the exact solutions of NSDDEs \eqref{eq:NSDDEsdiffform} is needed.

\begin{lemma}\label{lem:holder continuity}
      Suppose that Assumptions \ref{ass:mainassone} and \ref{ass:mainasstwo} hold. If $\phi \in \mathbb{C}([-\tau,0];\R^{d})$ satisfies the global Lipschitz continuous condition, i.e., there exists $K > 0$ such that
      \begin{equation}\label{eq:phiLipschitz}
            |\phi(t)-\phi(s)|
            \leq
            K|t-s|,
            \quad
            s,t \in [-\tau,0],
      \end{equation}
      then for any $\delta \in [2,\frac{p^*}{q}]$ and  $0 \leq s < t \leq T $ with $t-k\tau \in [-\tau,0], s-k\tau \in [-\tau,0], k \in \N$, there exists $C > 0$, depending on $T, \phi$ and $k$, such that
      \begin{equation}\label{eq:holder continuity}
            |X(t)-X(s)|_{L^{\delta}(\Omega;\R^d)}
            \leq
            C\big((t-s)+(t-s)^\frac{1}{2}\big).
      \end{equation}
\end{lemma}

\begin{proof}
      Owing to \eqref{eq:NSDDEsintegral} and the Minkowski inequality, we have
      \begin{align*}
            |X(t)-X(s)| _{L^{\delta}(\Omega;\R^d)}
            \leq&~
            |D(X(t-\tau))-D(X(s-\tau))|_{L^{\delta}(\Omega;\R^d)}
            \notag
			\\&~+
			\Big|\int_{s}^{t} b(r,X(r),X(r-\tau))\diff{r}
            \Big|_{L^{\delta}(\Omega;\R^d)}
            \notag
            \\&~+
            \Big|\int_{s}^{t} \sigma(r,X(r),X(r-\tau))\diff{W(r)}
            \Big|_{L^{\delta}(\Omega;\R^d)}.
      \end{align*}
      It follows from \eqref{eq:DCompression}, \eqref{eq:bsuperlinear}, the H\"{o}lder inequality and the Burkholder--Davis--Gundy inequality that
      \begin{align*}
            &~|X(t)-X(s)| _{L^{\delta}(\Omega;\R^d)}
            \\\leq&~
            \nu|X(t-\tau)-X(s-\tau)|_{L^{\delta}(\Omega;\R^d)}
			+
			\int_{s}^{t} \big|b(r,X(r),X(r-\tau))
            \big|_{L^{\delta}(\Omega;\R^d)}\diff{r}
            \\&~+
            C\bigg((t-s)^{\frac{\delta-2}{2}}
            \int_{s}^{t} \big|\sigma(r,X(r),X(r-\tau))
            \big|_{L^{\delta}(\Omega;\R^d)}^{\delta}\diff{r}
            \bigg)^{\frac{1}{\delta}}
            \\\leq&~
            \nu|X(t-\tau)-X(s-\tau)|_{L^{\delta}(\Omega;\R^d)}
			+
			C\int_{s}^{t} 1 + |X(r)|_{L^{q\delta}(\Omega;\R^d)}^{q}
            + |X(r-\tau)|_{L^{q\delta}(\Omega;\R^d)}^{q} \diff{r}
            \\&~+
            C\bigg((t-s)^{\frac{\delta-2}{2}}
            \int_{s}^{t} 1 + |X(r)|_{L^{q\delta}(\Omega;\R^d)}^{q\delta}
            + |X(r-\tau)|_{L^{q\delta}(\Omega;\R^d)}^{q\delta} \diff{r}
            \bigg)^{\frac{1}{\delta}}.
      \end{align*}
      By \eqref{eq:Xmomentbound} and $q\delta \leq p^{*}$, one gets
      \begin{equation*}
            |X(t)-X(s)|_{L^{\delta}(\Omega;\R^d)}
            \leq
            C\Big((t-s) + (t-s)^\frac{1}{2} +
            |X(t-\tau)-X(s-\tau)|_{L^{\delta}(\Omega;\R^d)}\Big).
      \end{equation*}
      Applying iterative arguments leads to
      \begin{align*}
            &~|X(t)-X(s)|_{L^{\delta}(\Omega;\R^d)}
            \leq
            C \Big( (t-s)+(t-s)^\frac{1}{2}
            \\&~+
            C\big((t-s)+(t-s)^{\frac{1}{2}}
            +
            |X(t-2\tau)-X(s-2\tau)|_{L^{\delta}(\Omega;\R^d)}\big)\Big)
			\\\leq&~\cdots
			\\\leq&~
			(C + C^{2} + \cdots + C^{k})
            \big((t-s)+(t-s)^\frac{1}{2}\big)
            \\&~+
            C^{k}|X(t-k\tau)-X(s-k\tau)|_{L^{\delta}(\Omega;\R^d)}
			\\\leq&~
            C\big((t-s)+(t-s)^\frac{1}{2}
            +
            |\phi(t-k\tau)-\phi(s-k\tau)|\big)
            \\\leq&~
			C\big((t-s) + (t-s)^\frac{1}{2} + (t-s)\big)
			\\\leq&~
			C\big((t-s)+(t-s)^\frac{1}{2}\big),
      \end{align*}
      where \eqref{eq:phiLipschitz} has been used.
      Thus we complete the proof.
\end{proof}

Now we are able to derive the mean square convergence rate of \eqref{eq:BEmethodintro} on any finite time $[-\tau, T]$ with the help of the previously established upper mean square error bound. Besides, we would like to mention that the convergence result in Theorem \ref{th:mul-noise-rate} still holds for the backward Euler method applied to SDDEs, i.e., NSDDEs \eqref{eq:NSDDEsdiffform} with the neutral term $D$ vanishing.

\begin{theorem}\label{th:mul-noise-rate}
      Suppose that $\phi \in \mathbb{C}([-\tau,0];\R^{d})$ satisfies the global Lipschitz continuous condition \eqref{eq:phiLipschitz} and that Assumptions \ref{ass:mainassone} and \ref{ass:mainasstwo} hold with $\nu + 2L\Delta \leq \rho$ for some $\rho \in (\nu,1)$. Then there exists $C > 0$, independent of $\Delta$, such that for all $n = 1,2,\cdots,N$,
      \begin{equation}\label{eq:covorder}
            \E\big[|X(t_n)-X_n|^{2}\big]
            \leq
			C \Delta.
      \end{equation}
\end{theorem}

\begin{proof}
      In view of Lemma \ref{upper error bounds}, it suffices to estimate $\E\big[|R_i|^2\big]$ and $\E\big[|\E({R_i}~|~ \F_{t_{i-1}})|^{2}\big]$ for $i = 1,2,\cdots,N$. We first note that \eqref{eq:NSDDEsintegral} and \eqref{eq:localerror} promise
      \begin{align}\label{eq:localerror integral}
            R_{i}
            =&~
            \int_{t_{i-1}}^{t_{i}}
            b(s,X(s),X(s-\tau))
            - b(t_{i},X(t_{i}),X(t_{i-M})) \diff{s}
            \notag
            \\&~+
            \int_{t_{i-1}}^{t_{i}}
            \sigma(s,X(s),X(s-\tau))
            - \sigma(t_{i-1},X(t_{i-1}),X(t_{i-1-M})) \diff{W(s)}.
      \end{align}
      It follows from the H\"{o}lder inequality and It\^{o} isometry that
      \begin{align*}
            \E\big[|R_{i}|^{2}\big]
            \leq&~
            2\E\bigg[\bigg|\int_{t_{i-1}}^{t_{i}}
            b(s,X(s),X(s-\tau))
            - b(t_{i},X(t_{i}),X(t_{i-M})) \diff{s}\bigg|^{2}\bigg]
            \\&~+
            2\E\bigg[\bigg|\int_{t_{i-1}}^{t_{i}}
            \sigma(s,X(s),X(s-\tau))
            - \sigma(t_{i-1},X(t_{i-1}),X(t_{i-1-M})) \diff{W(s)}
            \bigg|^{2}\bigg]
            \\\leq&~
            2\Delta\int_{t_{i-1}}^{t_{i}}\E\big[|
            b(s,X(s),X(s-\tau))
            - b(t_{i},X(t_{i}),X(t_{i-M}))|^{2}\big] \diff{s}
            \\&~+
            2\int_{t_{i-1}}^{t_{i}}\E\big[|
            \sigma(s,X(s),X(s-\tau))
            - \sigma(t_{i-1},X(t_{i-1}),X(t_{i-1-M}))|^{2}\big] \diff{s}.
      \end{align*}
      Utilizing \eqref{eq:bpolynomialgrow} leads to
      \begin{align*}
            \E\big[|R_{i}|^{2}\big]
      \leq&~
            C\Delta\int_{t_{i-1}}^{t_{i}}\E\big[
            (1 + |X(s)|^{q-1} + |X(s-\tau)|^{q-1}
            + |X(t_{i})|^{q-1} + |X(t_{i-M})|^{q-1})^{2}
            \\&~\times
            (|X(s)-X(t_{i})| + |X(s-\tau)-X(t_{i-M})|)^{2}\big] \diff{s}
            \\&~+
            C\int_{t_{i-1}}^{t_{i}}\E\big[
            (1 + |X(s)|^{q-1} + |X(s-\tau)|^{q-1}
            + |X(t_{i-1})|^{q-1} + |X(t_{i-1-M})|^{q-1})^{2}
            \\&~\times
            (|X(s)-X(t_{i-1})| + |X(s-\tau)-X(t_{i-1-M})|)^{2}\big] \diff{s}.
      \end{align*}
      Along with the H\"{o}lder inequality, \eqref{eq:Xmomentbound} and Lemma \ref{lem:holder continuity}, we deduce that
      \begin{align}\label{eq:ERi2}
            \E\big[|R_{i}|^{2}\big]
      \leq&~
            C\Delta\int_{t_{i-1}}^{t_{i}}\Big(\E\big[
            (1 + |X(s)|^{q-1} + |X(s-\tau)|^{q-1}
            \notag
            \\&~+ |X(t_{i})|^{q-1}
            + |X(t_{i-M})|^{q-1})^{\frac{4q-2}{q-1}}
            \big]\Big)^{\frac{2(q-1)}{4q-2}}
            \notag
            \\&~\times
            \Big(\E\big[(|X(s)-X(t_{i})|
            + |X(s-\tau)-X(t_{i-M})|)^{\frac{4q-2}{q}}\big]
            \Big)^{\frac{2q}{4q-2}} \diff{s}
            \notag
            \\&~+
            C\int_{t_{i-1}}^{t_{i}}\Big(\E\big[
            (1 + |X(s)|^{q-1} + |X(s-\tau)|^{q-1}
            \notag
            \\&~+
            |X(t_{i-1})|^{q-1}
            +
            |X(t_{i-1-M})|^{q-1})^{\frac{4q-2}{q-1}}
            \big]\Big)^{\frac{2(q-1)}{4q-2}}
            \notag
            \\&~\times
            \Big(\E\big[(|X(s)-X(t_{i-1})|
            + |X(s-\tau)-X(t_{i-1-M})|)^{\frac{4q-2}{q}}\big]
            \Big)^{\frac{2q}{4q-2}} \diff{s}
            \notag
            \\\leq&~
            C\Delta^{2}.
      \end{align}
      Concerning $\E\big[|\E({R_i}~|~ \F_{t_{i-1}})|^{2}\big]$, we apply \eqref{eq:localerror integral}, the conditional Jensen inequality and the previous arguments used for \eqref{eq:ERi2} to derive
      \begin{align}\label{eq:E-R-F}
            &~\E\big[|\E({R_i}~|~\F_{t_{i-1}})|^{2}\big]
            \notag
      \\=&~
			\E\bigg[\bigg| \E \bigg( \int_{t_{i-1}}^{t_i}
			b(s,X(s),X(s-\tau)) - b(t_i,X(t_i),X(t_{i-M})) \diff{s}
			~\Big|~ \F_{t_{i-1}}\bigg)\bigg|^{2}\bigg]
            \notag
			\\\leq&~
			\E\bigg[\E\bigg(\Big|\int_{t_{i-1}}^{t_i}
			b(s,X(s),X(s-\tau)) - b(t_i,X(t_i),X(t_{i-M})) \diff{s}
			\Big|^2 ~\Big|~ \F _{t_{i-1}} \bigg) \bigg]
            \notag
			\\=&~
			\E\bigg[\bigg|\int_{t_{i-1}}^{t_i}
			b(s,X(s),X(s-\tau)) - b(t_i,X(t_i),X(t_{i-M})) \diff{s}
			\bigg|^{2}\bigg]
            \notag
			\\\leq&~
			C\Delta^3.
      \end{align}
      As a consequence of \eqref{eq:ERi2}, \eqref{eq:E-R-F} and Lemma \ref{upper error bounds}, we obtain \eqref{eq:covorder} and complete the proof.
\end{proof}

\section{Mean square exponential stability}\label{section4}

This section devotes to discussing the exponential stability of NSDDEs \eqref{eq:NSDDEsdiffform} and the backward Euler method \eqref{eq:BEmethodintro} on the infinite time interval $[-\tau,+\infty)$ in the mean square sense. We will prove that the backward Euler method can reproduce the mean square exponential stability of the original NSDDEs. The mean square exponential stability states that the second order moments of any two solutions with different initial values  will tend to zero exponentially fast; see, e.g., \cite{mao1994exponential, mao2008stochastic, shaikhet2013lyapunov} for more details. Let us first give the precise definition of the mean square exponential stability.
\begin{definition}\label{def:expmsstab}
      The exact solution of \eqref{eq:NSDDEsdiffform} is said to be mean square exponentially stable if for any two solutions $\{X(t)\}_{ t\geq -\tau}$ and $\{\bar{X}(t)\}_{t \geq -\tau}$ with different initial values $\phi,\bar{\phi} \in \C (\left [-\tau,0 \right ] ;\R^d)$ respectively, there exists a constant $\omega > 0$ such that
	  \begin{equation}
			\varlimsup_{t \to \infty} \frac{1}{t}
            \ln\E\big[|X(t)-\bar{X}(t)|^2\big]
			\leq
			-\omega.
	  \end{equation}
\end{definition}

In order to develop the exponential stability of NSDDEs \eqref{eq:NSDDEsdiffform}, we put another version of generalized monotonicity condition.

\begin{assumption}\label{ass:stab}
      Assume that there exist constants $\zeta > 1, c_1 >c_2 >0, c_3 > c_4 \geq 0$ and function $V \in \mathcal{V}(\R^{d}\times\R^{d};\R^{+})$ such that
	  \begin{align}\label{eq:genmonostability}
            &~2\langle  (x_{1}-x_{2})-(D(y_{1})-D(y_{2})),
			b(t,x_{1},y_{1})-b(t,x_{2},y_{2}) \rangle
			+
			\zeta |\sigma(t,x_{1},y_{1})-\sigma(t,x_{2},y_{2})|^{2}
            \notag
			\\\leq&~
			- c_1 |x_{1}-x_{2}|^{2}
			+ c_2 |y_{1}-y_{2}|^{2}
			- c_3 V(x_1,x_2)
			+ c_4 V(y_1,y_2),
			\quad
            t \geq 0, x_1,x_2,y_1,y_2 \in \R^{d}.
	  \end{align}
\end{assumption}

The following result concerns the exponential stability in mean square sense of the origin equations \eqref{eq:NSDDEsdiffform}. To state clearly, in what follows we adopt the conventional definition $\frac{1}{0} = +\infty$ and use ${\bf{1}}_{\{c_{4} > 0\}}$ to denote $1$ for the case $c_{4} > 0$  and $+\infty$ for $c_{4} = 0$.

\begin{theorem}\label{eqsolution-stab}
      Suppose that Assumptions \ref{ass:mainassone} and \ref{ass:stab} hold. Then for any two different initial values $\phi,\bar{\phi} \in \C ([-\tau,0];\R^{d})$ of \eqref{eq:NSDDEsdiffform}, the corresponding solutions $\{X(t)\}_{t \geq -\tau}$ and $\{\bar{X}(t)\}_{t \geq -\tau}$ satisfy
      \begin{equation}\label{eq:exactexponential}
            \varlimsup_{t \to \infty} \frac{1}{t}
            \ln\E\big[|X(t)-\bar{X}(t)|^2 \big]
            \leq
            -\omega
      \end{equation}
      for any $\omega \in \big(0,\omega_{1} \wedge \omega_{2}
      \wedge \frac{1}{\tau}\ln (\frac{c_3}{c_4} {\bf{1}}_{\{c_{4} > 0\}})\big)$ with $\omega_{2} \in (0,\frac{1}{\tau}\ln\frac{1}{\nu})$, where $\omega_1$ is the unique root of
      \begin{equation*}
	        -c_1 +(1+\nu)\omega_1 + c_2 e^{\omega_1 \tau}
            + (1+\nu)\nu\omega_1 e^{\omega_1 \tau} = 0.
      \end{equation*}
\end{theorem}

\begin{proof}
      According to \eqref{eq:NSDDEsdiffform} and the It\^{o} formula (see, e.g., \cite[(3.9)]{luo2006newcriteria}), we have that for any $\varepsilon > 0$,
      \begin{align}\label{eq:stab_2}
                  &~ e^{\omega t}\E\big[|X(t) - D(X(t-\tau))
                  -
                  \bar{X}(t) + D(\bar{X}(t-\tau))|^2\big]
                  \notag
                  \\=&~
                  |\phi(0) - D(\phi(-\tau))
                  - \bar{\phi}(0) +  D(\bar{\phi}(-\tau))|^2
                  \notag
                  \\&~+
                  \omega \int_{0}^{t} e^{\omega s} \E\big[
                  | X(s) - D( X(s-\tau) )
                  -
                  \bar{X}(s) + D(\bar{X}(s-\tau))|^2\big] \diff{s}
                  \notag
                  \\&~+
                  \int_{0}^{t} e^{\omega s}
                  \E\big[\big( 2\big\langle
                  X(s) - D(X(s-\tau))
                  -
                  \bar{X}(s) + D( \bar{X}(s-\tau)),
                  \notag
                  \\&~b(s,X(s),X(s-\tau))
                  - b(s,\bar{X}(s),\bar{X}(s-\tau))
                  \big \rangle
                  \notag
                  \\&~+
                  |\sigma( s, X(s),X(s-\tau))
                  - \sigma(s,\bar{X}(s),\bar{X}(s-\tau))|^2\big] \big)\diff{s}
                  \notag
                  \\=:&~
                  |\phi(0) - D(\phi(-\tau))
                  - \bar{\phi}(0) +  D(\bar{\phi}(-\tau))|^2
                  + J_1 + J_2.
      \end{align}
      Noting that \cite[Lemma 4.1 in Chapter 6]{mao2008stochastic} shows that for any $a,b \in \R, \varepsilon > 0$ and $p > 1$,
      \begin{align*}
            |a+b|^{p}
            \leq
            \bigg(1 + \varepsilon^{\frac{1}{p-1}}\bigg)^{p-1}
            \bigg(|a|^{p} + \frac{|b|^{p}}{\varepsilon}\bigg),
      \end{align*}
      we take $p = 2, \varepsilon = \nu$ and use \eqref{eq:DCompression} to get
      \begin{align}\label{eq:stab_3}
                  J_1
                  \leq&~
                  \omega \int_{0}^{t} e^{\omega s}
                  \E\bigg[(1+\nu)\bigg(|X(s)-\bar{X}(s)|^{2}
                  \notag
                  \\&~+
                  \frac{|D(X(s-\tau))-D(\bar{X}(s-\tau))|^2}{\nu}
                  \bigg)\bigg] \diff{s}
                  \notag
                  \\\leq&~
                  \omega(1+\nu)\int_{0}^{t} e^{\omega s}
                  \E[|X(s)-\bar{X}(s)|^2]\diff{s}
                  \notag
                  \\&~+
                  \omega\nu(1+\nu) \int_{0}^{t} e^{\omega s}
                  \E[|X(s-\tau)-\bar{X}(s-\tau)|^2]\diff{s}.
      \end{align}
      Applying \eqref{ass:stab} shows
      \begin{align}\label{eq:stab_4}
                  J_2
                  \leq&~
                  \int_{0}^{t} e^{\omega s}
                  \big(-c_1\E[|X(s)-\bar{X}(s)|^2]
                  +
                  c_2 \E[|X(s-\tau)-\bar{X}(s-\tau)|^2]
                  \notag
                  \\&~-
                  c_3 \E[V(X(s),\bar{X}(s))]
                  +
                  c_4\E[V(X(s-\tau),\bar{X}(s-\tau))]\big)\diff{s}.
	  \end{align}
      It follows from \eqref{eq:stab_2}, \eqref{eq:stab_3} and  \eqref{eq:stab_4} that
      \begin{align}\label{eq:stab_5}
                  &~ e^{\omega t}\E\big[
                  | X(t) - D( X(t-\tau) )
                  -
                  \bar{X}(t) + D( \bar{X}(t-\tau) ) |^2 \big]
                  \notag
                  \\\leq&~
                  |\phi(0) - D(\phi(-\tau))
                  - \bar{\phi}(0) + D(\bar{\phi}(-\tau))|^2
                  \notag
                  \\&~+
                  (\omega(1+\nu)-c_1) \int_{0}^{t}
                  e^{\omega s} \E[|X(s)-\bar{X}(s)|^2] \diff{s}
                  \notag
                  \\&~+
                  (\omega\nu(1+\nu)+c_2) \int_{0}^{t} e^{\omega s}
                  \E[|X(s-\tau)-\bar{X}(s-\tau)|^2] \diff{s}
                  \notag
                  \\&~-
                  c_3 \int_{0}^{t} e^{\omega s}
                  \E[V(X(s),\bar{X}(s))] \diff{s}
                  +
                  c_4 \int_{0}^{t} e^{\omega s}
                  \E[V(X(s-\tau),\bar{X}(s-\tau))] \diff{s}.
      \end{align}
      Owing to
      \begin{align*}
                  &~\int_{0}^{t} e^{\omega s}
                  \E[|X(s-\tau)-\bar{X}(s-\tau)|^{2}] \diff{s}
					\\=&~
					e^{\omega\tau} \int_{-\tau}^{t-\tau} e^{\omega s}
                    \E[|X(s) - \bar{X}(s)|^{2}] \diff{s}
					\\\leq&~
					e^{\omega\tau} \int_{-\tau}^{0}
					|\phi(s) - \bar{\phi}(s) | ^ 2 \diff{s}
					+
					e^{\omega \tau} \int_{0}^{t} e^{\omega s}
                    \E[|X(s)-\bar{X}(s)|^{2}] \diff{s},
	  \end{align*}
      and similarly
	  \begin{align*}
				&~\int_{0}^{t} e^{\omega s}
                \E[V(X(s-\tau),\bar{X}(s-\tau))] \diff{s}
				\\\leq&~
				e^{\omega \tau} \int_{-\tau}^{0}
				V(\phi(s) , \bar{\phi}(s)) \diff{s}
				+
				e^{\omega \tau} \int_{0}^{t}
				e^{\omega s} \E[V(X(s),\bar{X}(s))] \diff{s},
	  \end{align*}
      we utilize \eqref{eq:stab_5} to show
      \begin{align}\label{eq:stab_8}
					&~ e^{\omega t}\E\big[|X(t)-D(X(t-\tau))
					-
					\bar{X}(t) + D(\bar{X}(t-\tau))|^2\big]\notag
					\\\leq&~
					|\phi(0) - D(\phi(-\tau)) - \bar{\phi}(0)
                    +  D(\bar{\phi}(-\tau)) | ^ 2
					+
					c_4 e^{\omega \tau} \int_{-\tau}^{0}
					V( \phi(s) , \bar{\phi}(s) ) \diff{s} \notag
					\\&~+
					(\omega\nu(1+\nu)+c_2)
					e^{\omega \tau} \int_{-\tau}^{0}
					| \phi(s) - \bar{\phi}(s) | ^ 2 \diff{s} \notag
					\\&~+
					H_1(\omega) \int_{0}^{t} e^{\omega s}
					\E[|X(s)-\bar{X}(s)|^{2}] \diff{s}
					+
					H_2(\omega)
					\int_{0}^{t} e^{\omega s}
                    \E[V(X(s),\bar{X}(s))] \diff{s}
      \end{align}
      with
      \begin{align}\label{H_1}
			H_1(\omega)
			:=
			\omega(1+\nu)-c_1  +
            \omega\nu(1+\nu) e^{\omega \tau}
            + c_2 e^{\omega \tau},
            \quad
            H_2(\omega)
            :=
            c_4 e^{\omega \tau} - c_3.
      \end{align}
      In view of $c_1 > c_2 >0$, we have $H_1(0) = -c_1+c_2 < 0$ and $ H_1(\frac{1}{\tau}\ln\frac{c_1}{c_2}) > -c_1 + c_2 e^{(\frac{1}{\tau}\ln\frac{c_1}{c_2}) \tau} = 0$. Together with the fact $H_1^{'}(\omega) > 0$ for all $\omega > 0$, we deduce that there exists a unique constant $\omega_1 \in (0,\frac{1}{\tau}\ln\frac{c_1}{c_2})$ such that $H_1(\omega_1) = 0$ and consequently
      \begin{align}\label{eq:411411}
            H_1(\omega) \leq 0,
            \quad \omega \in (0,\omega_1].
      \end{align}
      Concerning $H_{2}$, we have
      \begin{align}\label{eq:412412}
            H_2(\omega) < 0,
            \quad
            \omega > 0
      \end{align}
      when $c_3 > c_4 = 0$. For the case $c_3 > c_4 > 0$, we use $H_2(0) = -c_3+c_4 < 0$ and $H_2(\frac{1}{\tau}\ln \frac{c_3}{c_4}) = 0$ to see that $H_2(\omega) \leq 0$ for any $\omega \in (0,\frac{1}{\tau}\ln \frac{c_3}{c_4}]$, which along with \eqref{eq:411411} and \eqref{eq:412412} implies
      \begin{align}\label{eq:H1H2leqzero}
            H_1(\omega) \leq 0,
            \quad
            H_2(\omega) \leq 0,
            \quad
            \omega \in \Big(0,\omega_1
            \wedge \frac{1}{\tau}\ln
            \Big(\frac{c_3}{c_4}{\bf{1}}_{\{c_{4} > 0\}}\Big)\Big].
      \end{align}
      Inserting \eqref{eq:H1H2leqzero} into \eqref{eq:stab_8} and using $\omega_1 < \frac{1}{\tau}\ln\frac{c_1}{c_2}$ shows that for any $\omega \in \big(0,\omega_1 \wedge \frac{1}{\tau}\ln (\frac{c_3}{c_4} {\bf{1}}_{\{c_{4} > 0\}})\big]$,
      \begin{align}\label{eq:stab_9}
					&~ e^{\omega t}\E\big[|X(t)-D( X(t-\tau))
					- \bar{X}(t) + D( \bar{X}(t-\tau))|^2 \big]\notag
                    \\\leq&~
					|\phi(0) - D(\phi(-\tau)) - \bar{\phi}(0)
                    +  D(\bar{\phi}(-\tau)) | ^ 2
					+
					c_3 \int_{-\tau}^{0}
					V( \phi(s) , \bar{\phi}(s) ) \diff{s} \notag
					\\&~+
					\bigg(\frac{\nu(1+\nu)}{\tau}
                    \ln\frac{c_1}{c_2} + c_2\bigg)
					\frac{c_{1}}{c_{2}} \int_{-\tau}^{0}
					|\phi(s)-\bar{\phi}(s)|^{2} \diff{s}
                    =:\hat{C}.
      \end{align}
      By \eqref{eq:DCompression} and the weighted Young inequality $ 2ab\leq \varepsilon a^2 + \frac{b^2}{\varepsilon} $ for all $a,b \in \R$ with $\varepsilon = \frac{1-\nu}{\nu} > 0$, we obtain
      \begin{align*}
            |X(t)-\bar{X}(t)|^{2}
            =&~
            |X(t)- D(X(t-\tau)) - \bar{X}(t) + D(\bar{X}(t-\tau))|^{2}
            \\&~+
            2\langle X(t)- D(X(t-\tau)) - \bar{X}(t) + D(\bar{X}(t-\tau)),
            \\&~
            D(X(t-\tau)) - D(\bar{X}(t-\tau))\rangle
            +
            |D(X(t-\tau)) - D(\bar{X}(t-\tau))|^{2}
            \\\leq&~
            (1+\varepsilon^{-1})
            |X(t)- D(X(t-\tau)) - \bar{X}(t) + D(\bar{X}(t-\tau))|^{2}
            \\&~+
            (1+\varepsilon)
            |D(X(t-\tau)) - D(\bar{X}(t-\tau))|^{2}
            \\\leq&~
            \frac{|X(t) - D(X(t-\tau)) - \bar{X}(t)
            + D(\bar{X}(t-\tau))|^2}{1-\nu}
            \\&~+
            \nu|X(t-\tau)-\bar{X}(t-\tau)|^{2}.
      \end{align*}
      In combination with \eqref{eq:stab_9}, we have that for any $\omega \in \big(0,\omega_1 \wedge \frac{1}{\tau}\ln (\frac{c_3}{c_4} {\bf{1}}_{\{c_{4} > 0\}})\big]$,
      \begin{align*}
            e^{\omega t}\E\big[|X(t)-\bar{X}(t)|^{2}\big]
            \leq
            \frac{\hat{C}}{1-\nu}
            +
            \nu \E\big[e^{\omega t}|X(t-\tau)
            - \bar{X}(t-\tau)|^{2}\big],
      \end{align*}
      which indicates that for any $T > 0$,
      \begin{align}\label{eq:stab_12}
            \sup_{0 \leq t \leq T} e^{\omega t}
            \E\big[|X(t) - \bar{X}(t)|^{2}\big]
            \leq&~
            \frac{\hat{C}}{1 - \nu}
            +
            \nu e^{\omega \tau}\sup_{0 \leq t \leq T}
            e^{\omega (t-\tau)} \E\big[
            |X(t-\tau) - \bar{X}(t-\tau)|^{2}\big]
            \notag
            \\\leq&~
            \frac{\hat{C}}{1-\nu}
            +
            \nu e^{\omega \tau} \sup_{-\tau \leq t \leq 0}
            \E\big[|\phi(t)-\bar{\phi}(t)|^{2}\big]
            \notag
            \\&~+
            \nu e^{\omega \tau} \sup_{0\leq t\leq T}
            e^{\omega t} \E\big[|X(t)-\bar{X}(t)|^{2}\big].
      \end{align}
      Observing that for any $\omega_{2} \in (0,\frac{1}{\tau}\ln\frac{1}{\nu})$ and any
      $\omega \in \big(0,\omega_{1} \wedge \omega_{2}
      \wedge \frac{1}{\tau}\ln (\frac{c_3}{c_4} {\bf{1}}_{\{c_{4} > 0\}})\big)$, one can use $0 < \nu e^{\omega\tau} \leq
      \nu e^{\omega_{2}\tau} < 1$ and \eqref{eq:stab_12} to derive
      \begin{align*}
            \sup_{0 \leq t \leq T} e^{\omega t}
            \E\big[|X(t)-\bar{X}(t)|^2 \big]
            \leq
            \frac{1}{1-\nu e^{\omega_{2}\tau}}
            \bigg(\frac{\hat{C}}{1-\nu}
            +
            \sup_{-\tau \leq t \leq 0}
            \E\big[|\phi(t)-\bar{\phi}(t)|^{2}\big]\bigg).
      \end{align*}
      By taking the limit $T \to \infty$, we see that for any
      $\omega \in \big(0,\omega_{1} \wedge \omega_{2}
      \wedge \frac{1}{\tau}\ln (\frac{c_3}{c_4} {\bf{1}}_{\{c_{4} > 0\}})\big)$ and any $t \geq 0$,
      \begin{align*}
            \E\big[|X(t)-\bar{X}(t)|^2 \big]
            \leq
            \frac{e^{-\omega t}}{1-\nu e^{\omega_{2}\tau}}
            \bigg(\frac{\hat{C}}{1-\nu}
            +
            \sup_{-\tau \leq t \leq 0}
            \E\big[|\phi(t)-\bar{\phi}(t)|^{2}\big]\bigg),
      \end{align*}
      which immediately implies \eqref{eq:exactexponential} and completes the proof.
\end{proof}

Due to the fact that Assumption \ref{ass:stab} implies Assumption \ref{ass:mainasstwo}, repeating the proof of Lemma \ref{well-posedness-num}
shows that \eqref{eq:BEmethodintro} admits a unique solution $\{X_{n}\}_{n \geq -M}$ with probability one for any $\Delta > 0$. Now one can follow Definition \ref{def:expmsstab} to define the exponential stability in mean square sense for numerical solutions generated by a numerical method.

\begin{definition}
      The numerical solution generated by \eqref{eq:BEmethodintro} is said to be mean square exponentially stable if for any two solutions $\{X_n\}_{n \geq -M}$ and $\{\bar{X}_n\}_{n \geq -M}$ with different initial values $\phi,\bar{\phi} \in \C(\big[-\tau,0\big];\R^{d})$ respectively, there exists a constant $\omega > 0$ such that
      \begin{equation}
	        \varlimsup_{n \to \infty}\frac{1}{t_{n}}
            \ln\E\big[|X_n-\bar{X}_n|^2\big]
            \leq
            -\omega.
      \end{equation}
\end{definition}

We finally prove that the backward Euler method can inherit the mean square exponential stability of the original equations.

\begin{theorem}\label{th:bemsolution-stab}
      Suppose that Assumptions \ref{ass:mainassone} and \ref{ass:stab} hold and $\Delta \in (0,\log_{(\frac{c_3}{c_4}{\bf{1}}_{\{c_{4} > 0\}})^{\frac{1}{\tau}} \wedge (\frac{c_1}{c_2})^{\frac{1}{\tau}}}\zeta)$. Then for any two different initial values $\phi, \bar{\phi} \in \C(\big[-\tau,0\big];\R^{d})$ of \eqref{eq:NSDDEsdiffform}, the corresponding numerical solutions $\{X_n\}_{n \geq -M}$ and $\{\bar{X}_n\}_{n \geq -M}$ generated by  \eqref{eq:BEmethodintro} satisfy
      \begin{equation}\label{eq:numericalstability}
            \varlimsup_{n \to \infty} \frac{1}{t_{n}}
            \ln\E\big[|X_{n}-\bar{X}_n|^{2}\big]
            \leq
            -\omega
      \end{equation}
      for any $\omega \in \big(0,\bar{\omega} \wedge \ln \big(\big(\frac{c_3}{c_4}{\bf{1}}_{\{c_{4} > 0\}}\big)^{\frac{1}{\tau}} \wedge C^{*}\big)\big)$ with $\bar{\omega} \in (0,\frac{2}{\tau}\ln\frac{1}{\nu})$, where $C^{*} \in (1, (\frac{c_1}{c_2})^{\frac{1}{\tau}})$ is the unique root of
      \begin{align*}
            (1+\nu)\frac{C^{\Delta }-1}
            {\Delta}-c_1
			+
			(1+\nu)\nu
			\frac{C^{\Delta}-1}{\Delta}C^{\tau}
            +
            c_2 C^{\tau} = 0.
      \end{align*}
\end{theorem}

\begin{proof}
      Since $\{X_n\}_{n \geq -M}$ and  $\{\bar{X}_n\}_{n \geq -M}$ are generated by \eqref{eq:BEmethodintro} with different initial values $\phi_n$ and $\bar{\phi}_n$ respectively, we have $X_{n} = \phi(t_{n}), \bar{X}_{n} = \bar{\phi}(t_{n})$ for $n = -M,-M+1,\cdots,0$ and
      \begin{align}\label{eq:BEstability-}
            &~X_{n} - D(X_{n-M}) - \bar{X}_{n} + D(\bar{X}_{n-M})
            \notag
            \\=&~
            X_{n-1}-D(X_{n-1-M})
            -
            \bar{X}_{n-1} + D(\bar{X}_{n-1-M})
            \notag
            \\&~+
            \big(b(t_{n},X_{n},X_{n-M})
            -
            b(t_{n},\bar{X}_{n},\bar{X}_{n-M})\big)\Delta
            \notag
            \\&~+
            \big(\sigma(t_{n-1},X_{n-1},X_{n-1-M})
            -
            \sigma(t_{n-1},\bar{X}_{n-1},\bar{X}_{n-1-M})\big)
            \Delta{W_{n-1}}, \quad n = 1,2,\cdots.
      \end{align}
      Adopting the notations $\e_{n}:=X_{n}-\bar{X}_{n}$, $\Delta\D_{n}:=D(X_{n-M})-D(\bar{X}_{n-M})$ and
      \begin{gather*}
            \Delta\bb_{n}:=b(t_{n},X_{n},X_{n-M})
            -
            b(t_{n},\bar{X}_{n},\bar{X}_{n-M}),
            \quad
            \Delta\si_{n}:=\sigma(t_{n},X_{n},X_{n-M})
            -
            \sigma(t_{n},\bar{X}_{n},\bar{X}_{n-M})
      \end{gather*}
      helps us to get
      \begin{equation}\label{key}
            \e_n - \Delta\D_{n}
            =
            \e_{n-1} - \Delta\D_{n-1} + \Delta\bb_{n} \Delta
            +
            \Delta\si_{n-1} \Delta W_{n-1},
            \quad n = 1,2,\cdots.
      \end{equation}
      It follows from \eqref{eq:identity} and $\E\big[\langle \e_{n-1}-\Delta\D_{n-1}, \Delta\si_{n-1}\Delta W_{n-1} \rangle\big] = 0$ that for any $n = 1,2,\cdots$,
      \begin{align}\label{stab_T}
            &~\E\big[|\e_n-\Delta\D_{n}|^2\big]
            -
            \E\big[|\e_{n-1}-\Delta\D_{n-1}|^2\big]
            +
            \E\big[|\e_n-\Delta\D_{n}
            -
            \e_{n-1}+\Delta\D_{n-1}|^2\big]
            \notag
            \\=&~
            2\Delta\E\big[\langle \e_n-\Delta\D_{n},
            \Delta\bb_{n} \rangle\big]
            +
            2\E\big[\langle \e_n-\Delta\D_{n},
            \Delta\si_{n-1}\Delta W_{n-1}\rangle\big]
            \notag
            \\=&~
            2\Delta\E\big[\langle \e_n-\Delta\D_{n},
            \Delta\bb_{n} \rangle\big]
            +
            2\E\big[\langle \Delta\si_{n-1}\Delta W_{n-1},
            \e_n - \Delta\D_{n} - \e_{n-1}
            + \Delta\D_{n-1} \rangle\big].
      \end{align}
      By Assumption \ref{ass:stab} and
      $\E\big[|\Delta \si_{n-1} \Delta W_{n-1}|^2\big]
      =\Delta\E\big[|\Delta\si_{n-1}|^2\big]$, we obtain
      \begin{align}
            &~\E\big[|\e_n-\Delta\D_{n}|^2\big]
            -
            \E\big[|\e_{n-1}-\Delta\D_{n-1}|^2\big]
            +
            \E\big[|\e_n-\Delta\D_{n}
            -
            \e_{n-1}+\Delta\D_{n-1}|^2\big]
            \notag
            \\\leq&~
            -c_1 \Delta\E\big[|\e_n|^2\big]
            +
            c_2 \Delta\E\big[|\e_{n-M}|^2\big]
            -
            c_3 \Delta\E\big[V(X_{n},\bar{X}_n)\big]
            \notag
            +
            c_4 \Delta\E\big[V(X_{n-M},\bar{X}_{n-M})\big]
            \\&~-
            \zeta\Delta\E\big[|\Delta\si_{n}|^2\big]
            +
            \Delta\E\big[|\Delta\si_{n-1}|^2\big]
            +
            \E\big[|\e_n-\Delta\D_{n}
            -
            \big(\e_{n-1}-\Delta\D_{n-1}\big)|^2\big],
            \notag
      \end{align}
      and consequently
      \begin{align}\label{eq:nstab_1}
            &~\E\big[|\e_n-\Delta\D_{n}|^2\big]
            +
            \zeta\Delta \E\big[|\Delta\si_{n}|^2\big]
            \notag
            \\\leq&~
            \E\big[|\e_{n-1}-\Delta\D_{n-1}|^2\big]
            +
            \Delta \E\big[|\Delta\si_{n-1}|^2\big]
            -
            c_1 \Delta\E\big[|\e_n|^2\big] \notag
            \\&~+
            c_2 \Delta \E\big[|\e_{n-M}|^2\big]
            -
            c_3 \Delta \E\big[V(X_{n},\bar{X}_n)\big]
            +
            c_4 \Delta \E\big[V(X_{n-M},\bar{X}_{n-M})\big].
      \end{align}
      For any $C \geq 1$, we multiply $C^{t_{n}}$ on the both sides of \eqref{eq:nstab_1} and subtract $C^{t_{n-1}} \E\big[|\e_{n-1}-\Delta\D_{n-1}|^2\big] + C^{t_{n-1}}\zeta\Delta \E\big[|\Delta\si_{n-1}|^2\big]$ to derive
      \begin{align*}
            &~ C^{t_{n}}\E\big[|\e_{n}-\Delta\D_{n}|^2\big]
            -
            C^{t_{n-1}}\E\big[|\e_{n-1}-\Delta\D_{n-1}|^2\big]
            +
            C^{t_{n}} \zeta\Delta
            \E\big[|\Delta\si_{n}|^2\big]
            -
            C^{t_{n-1}} \zeta\Delta
            \E\big[|\Delta\si_{n-1}|^2\big]
            \notag
            \\\leq&~
            (C^{\Delta}-1) C^{t_{n-1}}
            \E\big[|\e_{n-1}-\Delta\D_{n-1}|^2\big]
            +
            (C^{\Delta}-\zeta)\Delta C^{t_{n-1}}
            \E\big[|\Delta\si_{n-1}|^2\big] \notag
            -
            c_1 \Delta C^{t_{n}} \E\big[|\e_{n}|^2\big]
            \\&~+
            c_2 \Delta C^{t_{n}} \E\big[|\e_{n-M}|^2\big]
            -
            c_3 \Delta C^{t_{n}}
            \E \big[V( X_{n},\bar{X}_n)\big]
            +
            c_4 \Delta C^{t_{n}}
            \E\big[V(X_{n-M},\bar{X}_{n-M})\big].
      \end{align*}
      By iteration, one has
      \begin{align}\label{eq:nstab_8}
            &~ C^{t_{n}} \E\big[|\e_{n}-\Delta\D_{n}|^2\big]
            +
            C^{t_{n}} \zeta\Delta
            \E\big[|\Delta\si_{n}|^2\big]
            \notag
            \\\leq&~
            \E\big[|\e_{0}-\Delta\D_{0}|^2\big]
            +
            \zeta\Delta \E\big[|\Delta\si_{0}|^2\big]
            \notag
            +
            (C^{\Delta}-1)\sum_{i=0}^{n-1} C^{t_{i}}
            \E\big[|\e_{i}-\Delta\D_{i}|^2\big]
            \\&~+
            (C^{\Delta}-\zeta)\Delta\sum_{i=0}^{n-1}
            C^{t_{i}} \E\big[|\Delta\si_{i}|^2\big]
            -
            c_1 \Delta\sum_{i=1}^{n} C^{t_{i}}
            \E\big[|\e_{i}|^2\big]
            +
            c_2 \Delta\sum_{i=1}^{n} C^{t_{i}}
            \E\big[|\e_{i-M}|^2\big]
            \notag
            \\&~-
            c_3 \Delta\sum_{i=1}^{n} C^{t_{i}}
            \E\big[V( X_{i},\bar{X}_i)\big]
            +
            c_4 \Delta\sum_{i=1}^{n} C^{t_{i}}
            \E\big[V( X_{i-M},\bar{X}_{i-M} )\big].
      \end{align}
      Noting that \eqref{eq:DCompression} and the weighted Young inequality $2ab\leq\varepsilon a^2+\varepsilon^{-1}b^2$ for all $a,b \in \R$ with $\varepsilon=\nu$ imply
      \begin{align}
            \E\big[|\e_{i}-\Delta\D_{i}|^2\big]
            =&~
            \E[|\e_{i}|^2]
            +
            \E[|\Delta\D_{i}|^2]
            +
            2\E[\langle \e_{i},\Delta\D_{i}\rangle ]
            \notag
            \\\leq&~
            (1+\nu) \E[|\e_{i}|^2]
            +
            (1+\nu)\nu \E[|\e_{i-M}|^2],
            \notag
      \end{align}
      we combine this with \eqref{eq:nstab_8} to see that
      \begin{align}\label{eq:nstab_81}
            &~ C^{t_{n}} \E\big[|\e_{n}-\Delta\D_{n}|^2\big]
            +
			C^{t_{n}} \zeta\Delta
            \E \big[|\Delta\si_{n}|^2\big]
            \notag
			\\\leq&~
			\E\big[|\e_{0}-\Delta\D_{0}|^2\big]
			+
			\zeta\Delta \E\big[|\Delta\si_{0}|^2\big]
			+
			c_1 \Delta \E\big[|\e_{0}|^2\big]
            \notag
			\\&~+
			(C^{\Delta}-\zeta) \Delta\sum_{i=0}^{n-1}
            C^{t_{i}} \E\big[|\Delta\si_{i}|^2\big]
            \notag
			+
			\bigg((1+\nu)
			\frac{C^{\Delta}-1}{\Delta}-c_1\bigg)
			\Delta\sum_{i=0}^{n-1} C^{t_{i}}
			\E\big[|\e_{i}|^2\big]
            \notag
			\\&~+
			\bigg((1+\nu)\nu
            \frac{C^{\Delta} - 1}{\Delta}+c_2\bigg)
            \Delta\sum_{i=0}^{n} C^{t_{i}}
			\E\big[|\e_{i-M}|^2\big]
			+
			c_3 \Delta \E\big[V(X_{0},\bar{X}_0)\big]
            \notag
			\\&~-
			c_3 \Delta\sum_{i=0}^{n-1} C^{t_{i}}
			\E\big[V(X_{i},\bar{X}_i)\big]
			+
			c_4 \Delta\sum_{i=1}^{n} C^{t_{i}}
			\E\big[V(X_{i-M},\bar{X}_{i-M})\big].
      \end{align}
      Owing to
      \begin{align}\label{nstab_10}
            &~\sum_{i=0}^{n} C^{t_{i}}
			\E\big[|\e_{i-M}|^2\big]
			=
			\sum_{i=-M}^{n-M} C^{t_{i+M}}
			\E\big[|\e_{i}|^2\big]
            \notag
			\\=&~
			\sum_{i=-M}^{-1} C^{t_{i+M}}
			\E\big[|\e_{i}|^2\big]
			+
			\sum_{i=0}^{n-1} C^{t_{i+M}}
			\E\big[|\e_{i}|^2\big]
			-
			\sum_{i=n-M+1}^{n-1} C^{t_{i+M}}
			\E\big[|\e_{i}|^2\big]
            \notag
			\\\leq&~
			C^{\tau}\sum_{i=-M}^{-1} C^{t_{i}} |\e_{i}|^{2}
			+
			C^{\tau}\sum_{i=0}^{n-1} C^{t_{i}}
			\E\big[|\e_{i}|^2\big],
      \end{align}
      and
      \begin{align}\label{nstab_11}
            &~\sum_{i=1}^{n} C^{t_{i}}
			\E\big[V( X_{i-M},\bar{X}_{i-M})\big]
			=
			\sum_{i=-M+1}^{n-M} C^{t_{i+M}}
			\E\big[V( X_{i},\bar{X}_{i})\big]
            \notag
			\\=&~
			\sum_{i=-M+1}^{-1} C^{t_{i+M}}
			\E\big[V(X_{i},\bar{X}_{i})\big]
			+
			\sum_{i=0}^{n-1} C^{t_{i+M}}
			\E\big[V( X_{i},\bar{X}_{i})\big]
            \notag
			-
			\sum_{i=n-M+1}^{n-1} C^{t_{i+M}}
			\E\big[V( X_{i},\bar{X}_{i})\big]
            \notag
			\\\leq&~
			C^{\tau}\sum_{i=-M+1}^{-1} C^{t_{i}}
			V(X_{i},\bar{X}_{i})
			+
			C^{\tau}\sum_{i=0}^{n-1} C^{t_{i}}
			\E\big[V( X_{i},\bar{X}_{i})\big]
      \end{align}
      with the formal definition of summation $\sum\limits_{i = a}^{b} g(i) = 0$ for $b < a$ when $M = 1$, we utilize \eqref{eq:nstab_81} to obtain
      \begin{align}\label{eq:nstab_12}
            &~C^{t_{n}}\E\big[|\e_{n}-\Delta\D_{n}|^{2}\big]
			\leq
			C^{t_{n}}\E\big[|\e_{n}-\Delta\D_{n}|^{2}\big]
			+
            C^{t_{n}} \zeta\Delta
            \E\big[|\Delta\si_{n}|^{2}\big]
            \notag
            \\\leq&~
			\E\big[|\e_{0}-\Delta\D_{0}|^2\big]
			+
			\zeta\Delta \E\big[|\Delta\si_{0}|^2\big]
			+
			c_1 \Delta \E\big[|\e_{0}|^2\big]
			+
			c_3 \Delta \E\big[V(X_{0},\bar{X}_0)\big]
            \notag
			\\&~+
			\bigg((1+\nu)\nu
            \frac{C^{\Delta }-1}{\Delta} + c_2\bigg)
			\Delta C^{\tau} \sum_{i=-M}^{-1}
            C^{t_{i}} |\e_{i}|^{2} \notag
            \\&~+
            c_4 \Delta C^{\tau} \sum_{i=-M+1}^{-1}
            C^{t_{i}} V(X_{i},\bar{X}_{i})
            +
            h_1(C) \Delta\sum_{i=0}^{n-1} C^{t_{i}}\Delta
			\E\big[|\Delta\si_{i}|^2\big] \notag
            \\&~+
			h_2(C) \Delta\sum_{i=0}^{n-1} C^{t_{i}}
			\E\big[V(X_{i},\bar{X}_i)\big]
			+
			h_3(C) \Delta\sum_{i=0}^{n-1} C^{t_{i}}
            \E\big[|\e_{i}|^2\big]
      \end{align}
      with $h_1(C) := C^{\Delta}-\zeta$, $h_2(C) := c_4 C^{\tau} - c_3$ and
      \begin{align}\label{nstab_13}
			h_3(C)
			:=
            (1+\nu)\frac{C^{\Delta }-1}
            {\Delta}-c_1
			+
			(1+\nu)\nu
			\frac{C^{\Delta}-1}{\Delta}C^{\tau}
            +
            c_2 C^{\tau}.
      \end{align}
      For the case $c_3 > c_4 = 0$, we have
      \begin{align}\label{eq:h2C1}
            h_{2}(C) < 0, \quad C \geq 1.
      \end{align}
      For the case $c_3 > c_4 > 0$, we use $h_2(1) = -c_3+c_4 < 0$, $h_2((\frac{c_3}{c_4})^{\frac{1}{\tau}}) = 0$ and $h_{2}^{'}(C) > 0$ for all $C \geq 1$ to get $h_{2}(C) \leq 0$ for all $C \in \big[1, (\frac{c_3}{c_4})^{\frac{1}{\tau}}\big]$, which in combination with \eqref{eq:h2C1} promises
      \begin{align}\label{eq:h2C2}
            h_{2}(C) < 0,
            \quad C \in \bigg[1,
            \Big(\frac{c_3}{c_4}\Big)^{\frac{1}{\tau}}
            \textbf{1}_{\{c_{4} > 0\}}\bigg)
            =:
            [1,\hat{C}).
      \end{align}
      Observing $c_1 > c_2 > 0$, the facts $h_3(1) = -c_1+c_2 < 0$, $h_3((\frac{c_1}{c_2})^{\frac{1}{\tau}}) > -c_1 + c_2 ((\frac{c_1}{c_2})^{\frac{1}{\tau}})^{\tau}=0$ and $h_3^{'}(C) > 0$ for all $C \geq 1$ guarantee that there exists a unique constant $C^{*} \in (1, (\frac{c_1}{c_2})^{\frac{1}{\tau}})$ such that $h_3(C^{*})=0$ and therefore
      \begin{align}\label{eq:h3C}
            h_{3}(C) \leq 0,
            \quad C \in [1,C^{*}].
      \end{align}
      Due to $\Delta \in (0,\log_{(\frac{c_3}{c_4}{\bf{1}}_{\{c_{4} > 0\}})^{\frac{1}{\tau}} \wedge (\frac{c_1}{c_2})^{\frac{1}{\tau}}}\zeta)$, we use $\hat{C} \wedge C^{*} = \big(\frac{c_3}{c_4}{\bf{1}}_{\{c_{4} > 0\}}\big)^{\frac{1}{\tau}} \wedge C^{*}$ and take $C = e^{\omega} \in (1, \big(\frac{c_3}{c_4}{\bf{1}}_{\{c_{4} > 0\}}\big)^{\frac{1}{\tau}} \wedge C^{*})$ to derive
      \begin{align}\label{eq:h1C}
            h_{1}(e^{\omega})
            \leq
            \bigg(\Big(\frac{c_3}{c_4}{\bf{1}}_{\{c_{4} > 0\}}\Big)^{\frac{1}{\tau}} \wedge C^{*}\bigg)^{\log_{(\frac{c_3}{c_4}{\bf{1}}_{\{c_{4} > 0\}})^{\frac{1}{\tau}} \wedge (\frac{c_1}{c_2})^{\frac{1}{\tau}}}\zeta} - \zeta \leq 0
      \end{align}
      for all $\omega \in \big(0, \ln \big(\big(\frac{c_3}{c_4}{\bf{1}}_{\{c_{4} > 0\}}\big)^{\frac{1}{\tau}} \wedge C^{*}\big)\big)$
      It follows from \eqref{eq:nstab_12}, \eqref{eq:h2C1}, \eqref{eq:h2C2}, \eqref{eq:h3C} and \eqref{eq:h1C} that for any $n = 1,2,\cdots$,
      \begin{align*}
			e^{\omega t_{n}}
			\E\big[|\e_{n}-\Delta\D_{n}|^2\big]
			\leq&~
			\E\big[|\e_{0}-\Delta\D_{0}|^2\big]
			+
			\zeta\Delta \E\big[|\Delta\si_{0}|^2\big]
			+
			c_1 \Delta \E\big[|\e_{0}|^2\big]
            \notag
			\\&~+
			\bigg((1+\nu)\nu
			\frac{e^{\omega\Delta }-1}{\Delta}
            + c_2\bigg)
			e^{\omega \tau} \Delta \sum_{i=-M}^{-1}
            e^{\omega t_{i}} |\e_{i}|^{2}
            \notag
			\\&~+
			c_3 \Delta \E\big[V( X_{0},\bar{X}_0)\big]
            +
			c_4 e^{\omega \tau} \Delta
			\sum_{i=-M+1}^{-1} e^{\omega t_{i}}
			V(\phi_{i},\bar{\phi}_{i})
			=:
            C_{\phi,\bar{\phi}}.
      \end{align*}
      Together with  \eqref{eq:DCompression} and the weighted Young inequality $2ab \leq \epsilon a^2 + \epsilon^{-1}b^2$ for all $a,b \in \R$ with $\epsilon > 0$, we show that for all $\omega \in \big(0, \ln \big(\big(\frac{c_3}{c_4}{\bf{1}}_{\{c_{4} > 0\}}\big)^{\frac{1}{\tau}} \wedge C^{*}\big)\big)$ and all $i = 0,1,\cdots,n$,
      \begin{align}
			e^{\omega t_{i}} \E\big[|\e_{i}|^2\big]
			\leq&~
			(1+\epsilon) e^{\omega t_{i}}
			\E\big[|\e_{i}-\Delta\D_{i}|^2\big]
			+
			(1+\epsilon^{-1})\nu^2 e^{\omega t_{i}}
			\E\big[|\e_{i-M}|^2\big]
            \notag
			\\\leq&~
			(1+\epsilon) C_{\phi,\bar{\phi}}
			+
			(1+\epsilon^{-1})\nu^2 e^{\omega \tau}
			\sup_{-M\leq j \leq n} e^{\omega t_{j}}
			\E\big[|\e_{j}|^2\big], \notag
      \end{align}
      which also holds for all $i = -M,-M+1,\cdots,0$ and therefore implies
      \begin{align*}
			\sup_{-M\leq j \leq n} e^{\omega t_{j}}
			\E\big[|\e_{j}|^2\big]
			\leq
			(1+\epsilon) C_{\phi,\bar{\phi}}
			+
			(1+\epsilon^{-1})\nu^2 e^{\omega \tau}
			\sup_{-M\leq j \leq n} e^{\omega t_{j}}
			\E\big[|\e_{j}|^2\big].
      \end{align*}
      For any $\bar{\omega} \in (0,\frac{2}{\tau}\ln\frac{1}{\nu})$ and any $\omega \in \big(0,\bar{\omega} \wedge \ln \big(\big(\frac{c_3}{c_4}{\bf{1}}_{\{c_{4} > 0\}}\big)^{\frac{1}{\tau}} \wedge C^{*}\big)\big)$, we choose any $\epsilon > \frac{\nu^{2}e^{\bar{\omega}\tau}}{ 1-\nu^{2}e^{\bar{\omega}\tau}}$ to get $0 < (1+\epsilon^{-1})\nu^2 e^{\omega \tau} \leq (1+\epsilon^{-1})\nu^2 e^{\bar{\omega}\tau} < 1$ and
      \begin{align*}
            \sup_{-M\leq j \leq n}e^{\omega t_{j}}
			\E\big[|\e_{j}|^2\big]
			\leq
			\frac{(1+\epsilon)C_{\phi,\bar{\phi}}}
            {1 - (1+\epsilon^{-1})\nu^2 e^{\bar{\omega}\tau}}.
      \end{align*}
      Letting $n \to \infty$ indicates that for any $\omega \in \big(0,\bar{\omega} \wedge \ln \big(\big(\frac{c_3}{c_4}{\bf{1}}_{\{c_{4} > 0\}}\big)^{\frac{1}{\tau}} \wedge C^{*}\big)\big)$ and any $n = 1, 2,\cdots$,
      \begin{align*}
			e^{\omega t_{n}}
			\E\big[|\e_{n}|^2\big]
			\leq
			\frac{(1+\epsilon)C_{\phi,\bar{\phi}}}
            {1 - (1+\epsilon^{-1})\nu^2 e^{\bar{\omega}\tau}},
      \end{align*}
      which results in \eqref{eq:numericalstability} and finishes the proof.
\end{proof}

\section{Numerical experiments}\label{section5}

This section aims to perform some numerical experiments to illustrate the previous theoretical results, including the mean square convergence rate and the mean square exponential stability. Before doing this, we emphasize that in what follows the Newton--Raphson iterations with precision $10^{-5}$ are employed to solve the nonlinear equations arising from the implementation of the considered implicit method in every time step. Besides, the expectations are approximated by the Monte Carlo simulations with $1000$ different Brownian paths.


      Let us first consider the scalar NSDDE
      \begin{align}\label{eq:1DNSDDE}
            \diff{\bigg(X(t) + \frac{1}{2}X(t-\tau)\bigg)}
            =&~
            \bigg(\frac{1}{1+t} + X(t) - 5X^{3}(t)
             + 2\sin(X(t-\tau)) \bigg)\diff{t}
             \notag
             \\&~+
             \big(1 + \cos t + X^{2}(t)
             +|X(t-\tau)|\big) \diff{W(t)},
             \quad t\in(0,5]
      \end{align}
      with the initial value $\phi(t) = \cos t, t \in [-1,0]$. Here $\{W(t)\}_{t \in [0,5]}$ is a standard $\R$-valued Brownian motion on the complete filtered probability space $(\Omega,\F,\{\F_{t}\}_{t \in [0,5]},\P)$. In order to detect the mean square convergence rate of backward Euler method, the unavailable exact solution is identified with a numerical approximation generated by \eqref{eq:BEmethodintro} with a fine stepsize $\Delta = 2^{-14}$. The other numerical approximations are calculated by \eqref{eq:BEmethodintro} applied to \eqref{eq:1DNSDDE} with six different stepsizes $\Delta = 2^{-i}, i = 7,8,9,10,11,12$. From the left part of Figure \ref{fig:strongorder}, we see that the root mean square error line and the reference line appear to parallel to each other, showing that the mean square convergence rate of the backward Euler method is order $1/2$. A least square fit indicates that the slope of the line for the backward Euler method is $0.55$, identifying Theorem \ref{th:mul-noise-rate} numerically.

\begin{figure}[!htbp]
\begin{center}
      \subfigure[order simulation for \eqref{eq:1DNSDDE}]
      {\includegraphics[width=0.45\textwidth]{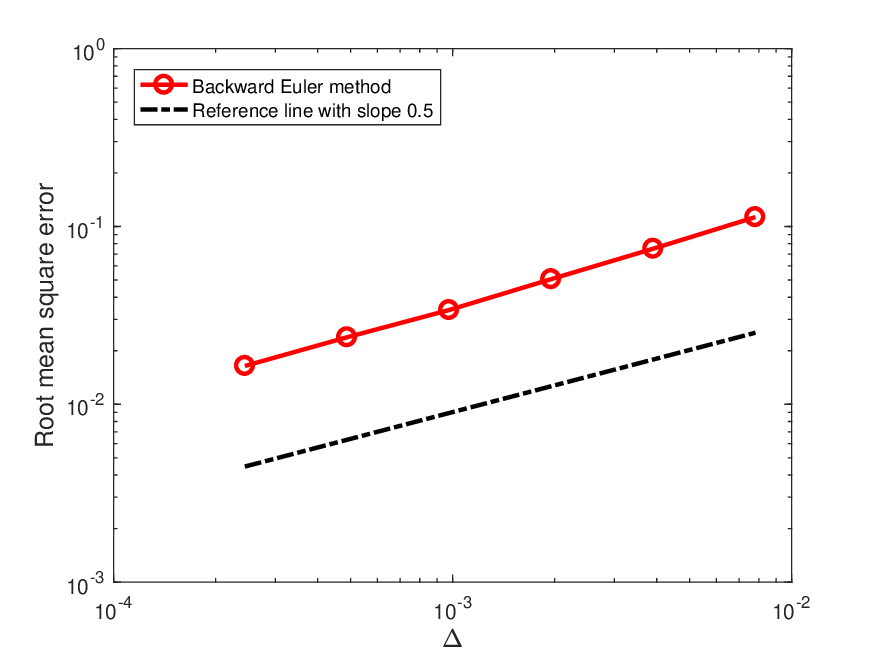}}
      \subfigure[order simulation for \eqref{eq:2DNSDDE}]
      {\includegraphics[width=0.45\textwidth]{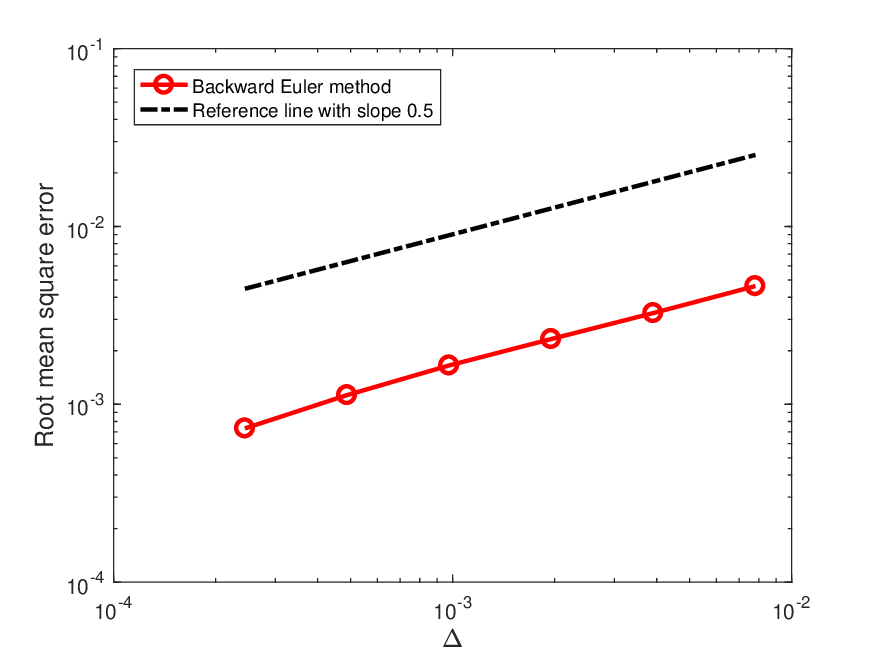}}
      \centering\caption{Mean square convergence rates of
      backward Euler method}
      \label{fig:strongorder}
\end{center}
\end{figure}

      We next focus on the following two-dimensional NSDDE
            \begin{align}\label{eq:2DNSDDE}
                  &~\diff{\left(
                  \begin{array}{cc}
                        \left[\begin{array}{cc}
                              X_1(t) \\ X_2(t)
                        \end{array}\right]
                        -
                        \frac{1}{2}\left[
                        \begin{array}{cc}
                              X_1(t-1) \\ X_2(t-1)
                        \end{array}\right]
                  \end{array}\right)} \notag
                  \\=&~\left(
                  \begin{array}{cc}
                        -\left(
                        \begin{array}{cc}
                              \left[\begin{array}{cc}
                                    4 & 0 \\ 0 & 4
                              \end{array}\right]
                              +
                              \left[\begin{array}{cc}
                                    t & 0 \\ 0 & t
                              \end{array}\right]
                        \end{array}\right)
                        \left[\begin{array}{cc}
                              X_1(t) \\ X_2(t)
                        \end{array}\right]
                        +
                        \left[\begin{array}{cc}
                              - 4X_1^3(t) \\ - 4X_2^3(t)
                        \end{array}\right]
                  \end{array}\right) \diff{t} \notag
                  \\&~+
                  \left[\begin{array}{cc}
                        X_1^2(t)+X_1(t-1) & 0 \\
                        0 & X_2^2(t)+X_2(t-1)
                  \end{array}\right]\diff{
                  \left[\begin{array}{cc}
                              W_{1}(t) \\ W_2(t)
                  \end{array}\right]},
                  \quad t \in (0,5]
            \end{align}
      with the initial value $(\phi_1(t),\phi_2(t))^{\top} = (\cos t,\sin t)^{\top}$ for $t \in [-1,0]$. Here $\{W_{1}(t)\}_{t \in [0,5]}$ and $\{W_{2}(t)\}_{t \in [0,5]}$ are two independent standard $\R$-valued Brownian motions on the complete filtered probability space $(\Omega,\F,\{\F_{t}\}_{t \in [0,5]},\P)$. It is worth noting that \eqref{eq:2DNSDDE} is a concrete example of the neutral stochastic cellular neural networks \cite{guo2013fixed}, which have been widely used in the fields of signal processing, pattern recognition, combinatorial optimization and so on. Under the same numerical setting of stepsizes for \eqref{eq:1DNSDDE}, the right part of Figure \ref{fig:strongorder} shows that the slopes of the error line and the reference line match well, indicating that the backward Euler method has the mean square convergence rate of order $1/2$.
      Additionally, a least square fit produces rate $0.52$. These facts further support the previous theoretical result in Theorem \ref{th:mul-noise-rate}.


\begin{figure}[!htbp]
      \centering
      \includegraphics[width=0.5\textwidth]{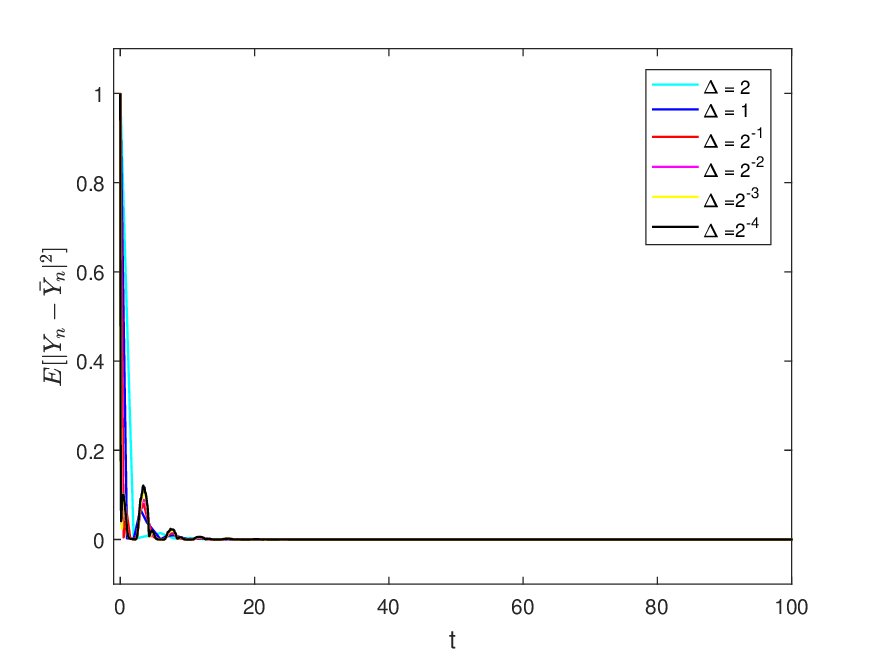}
      \caption{Mean square exponential stability of backward Euler method for \eqref{eq:examplestability}}
      \label{fig:NSDDE-stability}
\end{figure}

In order to test the exponential stability of numerical method, we modify Example 6.1 in \cite{deng2021tamed} as follows
      \begin{align}\label{eq:examplestability}
            &~\diff{\bigg(X(t)+\frac{1}{2}X(t-\tau)\bigg)} \notag
            \\=&~
            \bigg(-9X(t) + \frac{11}{2}X(t-\tau)
            -11\Big(X(t)+\frac{1}{2}X(t-\tau)\Big)
            \Big|X(t)+\frac{1}{2}X(t-\tau)\Big|
            \bigg)\diff{t} \notag
            \\&~+
            \bigg(X(t) + \sin(X(t-\tau))
            +
            \Big|X(t)+\frac{1}{2}X(t-\tau)
            \Big|^{\frac{3}{2}}\bigg)\diff{W(t)},
            \quad t > 0
      \end{align}
with two different initial values $\phi(t) = \cos t$ and $\bar{\phi}(t) = \sin t$ for $t \in [-4,0]$. Here $\{W(t)\}_{t\geq0}$ is a standard $\R$-valued Brownian motion on the complete filtered probability space $(\Omega,\F,\{\F_{t}\}_{t\geq0},\P)$. Noting that \eqref{eq:examplestability} tells $D(y) = -\frac{1}{2}y$,
\begin{gather*}
      b(t,x,y)
      =
      -9x + \frac{11}{2}y
      -
      11\Big(x+\frac{1}{2}y\Big)
      \Big|x+\frac{1}{2}y\Big|,
      \quad
      \sigma(t,x,y)
      =
      x + \sin y + \Big|x+\frac{1}{2}y\Big|^{\frac{3}{2}}
\end{gather*}
for all $t \geq 0$ and $x,y \in \R$ and we have $\nu = 1/2$. For any $x_{i}, y_{i} \in \R, i = 1,2$, we denote $Z_{i} := x_{i} + \frac{1}{2}y_{i}, i = 1,2$ to get
\begin{align}\label{eq:stablity-assu1}
      &~|b(t,x_{1},y_{1}) - b(t,x_{2},y_{2})|
      +
      |\sigma(t,x_{1},y_{1}) - \sigma(t,x_{2},y_{2})|\notag
      \\\leq&~
      20|x_{1}-x_{2}| + 11|Z_{1}-Z_{2}|
      + 11\big|Z_{1}|Z_{1}|-Z_{2}|Z_{2}|\big|\notag
      \\&~+
      |x_{1}-x_{2}| + |\sin y_{1} - \sin y_{2}|
      +
      \big||Z_{1}|^{\frac{3}{2}}
       - |Z_{2}|^{\frac{3}{2}}\big|\notag
      \\\leq&~
      21|x_{1}-x_{2}| + 11|Z_{1}-Z_{2}|
      + 11\big(|Z_{1}| + |Z_{2}|\big)|Z_{1}-Z_{2}|\notag
      \\&~+
       |y_{1}-y_{2}|
      +
      \big(|Z_{1}|^{\frac{1}{2}}
      + |Z_{2}|^{\frac{1}{2}}\big)
      |Z_{1} - Z_{2}|\notag
      \\\leq&~
      K(1 + |x_{1}| + |x_{2}| + |y_{1}| + |y_{2}|)
      (|x_{1}-x_{2}| + |y_{1}-y_{2}|),
\end{align}
where we have used the inequalities $\big|x|x|-y|y|\big| \leq (|x|+|y|)|x-y|$ and
\begin{align}\label{eq:55}
      \Big||x|^{\frac{3}{2}}-|y|^{\frac{3}{2}}\Big|
      \leq
      \big(|x|^{\frac{1}{2}}+|y|^{\frac{1}{2}}\big)
      \big||x|-|y|\big|
      \leq
      \big(|x|^{\frac{1}{2}}+|y|^{\frac{1}{2}}\big)
      \big|x-y\big|
\end{align}
for all $x,y \in \R$; see Appendix in \cite{sabanis2013euler}. It follows that \eqref{eq:bpolynomialgrow} holds with $q = 2$ and thus $p^{*} \geq 4q-2 = 6$. Taking $p^{*} = 6$ leads to
\begin{align}\label{eq:stablity-assu2}
      &~(1+|x|^{2})^{\frac{p^{*}}{2}-1}
      \big(2\langle x-D(y),b(t,x,y) \rangle
	  +
	  (p^*-1)|\sigma(t,x,y)|^{2}\big)\notag
      \\=&~
      (1+|x|^{2})^{2}
      \bigg(2\Big( x+\frac{1}{2}y\Big)
      \Big(-9x + \frac{11}{2}y\Big)
      -
      22\Big|x+\frac{1}{2}y\Big|^{3}\notag
	  +
	  5\bigg|x + \sin y + \Big|x +
      \frac{1}{2}y\Big|^{\frac{3}{2}}\bigg|^{2}\bigg)\notag
      \\\leq&~
      (1+|x|^{2})^{2}
      \bigg(-18x^{2} + \frac{13}{2}xy
      + \frac{11}{4}y^{2}
      -
      12\Big|x+\frac{1}{2}y\Big|^{3}
	  +
	  10|x + \sin y|^{2}\bigg)\notag
      \\\leq&~
      (1+|x|^{2})^{2}
      \bigg(-18x^{2} + \frac{13}{4}x^{2}
      + \frac{13}{4}y^{2} + \frac{11}{4}y^{2}
      -
      12\Big|x+\frac{1}{2}y\Big|^{3}
	  +
	  20|x|^{2} + 20|y|^{2}\bigg)\notag
      \\\leq&~
      (1+|x|^{2})^{2}
      \bigg(6x^{2} + 26|y|^{2}
      -
      12\Big|x+\frac{1}{2}y\Big|^{3}
	  \bigg)\notag
      \\\leq&~
	  K(1+|x|^{6}+|y|^{6})
\end{align}
for all $x, y \in \R$ with some constant $K > 0$, which is due to
\begin{align*}
      (1+|x|^{2})^{2}|x|^{2}
      \leq
      2|x|^{2} + 2|x|^{6}
      \leq
      2\bigg(\frac{(|x|^{2})^{\frac{6}{2}}}{\frac{6}{2}}
      + \frac{1^{\frac{6}{4}}}{\frac{6}{4}}\bigg) + 2|x|^{6}
      \leq
      3|x|^{6} + 2,
\end{align*}
and
\begin{align*}
      (1+|x|^{2})^{2}|y|^{2}
      \leq
      2|y|^{2} + |x|^{4}|y|^{2}
      \leq
      2\bigg(\frac{(|y|^{2})^{\frac{6}{2}}}{\frac{6}{2}}
      + \frac{1^{\frac{6}{4}}}{\frac{6}{4}}\bigg)
      +
      \frac{(|x|^{4})^{\frac{6}{4}}}{\frac{6}{4}}
      +
      \frac{(|y|^{2})^{\frac{6}{2}}}{\frac{6}{2}}
      \leq
      |x|^{6} + |y|^{6} +2.
\end{align*}
Then \eqref{eq:stablity-assu1} and \eqref{eq:stablity-assu2} support Assumption \ref{ass:mainassone}. Concerning Assumption \ref{ass:stab}, one can use $Z_{i} := x_{i} + \frac{1}{2}y_{i}, x_{i}, y_{i} \in \R, i = 1,2$ to show
\begin{align*}
      &~2\langle  (x_{1}-x_{2})-(D(y_{1})-D(y_{2})),
	  b(t,x_{1},y_{1})-b(t,x_{2},y_{2}) \rangle
	  \\&~+
	  \zeta |\sigma(t,x_{1},y_{1})
      -\sigma(t,x_{2},y_{2})|^{2}
     \\=&~
     2\bigg\langle  (x_{1}-x_{2})+\frac{1}{2}(y_{1}-y_{2}),
	 -20(x_{1}-x_{2}) \bigg\rangle
     \\&~+
     2\big\langle  Z_{1}-Z_{2},
	 \big( 11Z_{1} - 11Z_{1}|Z_{1}|\big)
      - \big( 11Z_{2} - 11Z_{2}|Z_{2}|\big) \big\rangle
	 \\&~+
	 \zeta
     |(x_{1} - x_{2}) + (\sin y_{1} - \sin y_{2})
      + (|Z_{1}|^{\frac{3}{2}}
       - |Z_{2}|^{\frac{3}{2}})|^{2}
     \\\leq&~
     -40|x_{1}-x_{2}|^{2}
	 -20(x_{1}-x_{2})(y_{1}-y_{2})
     +
     22|Z_{1}-Z_{2}|^{2}
     \\&~-
     22\big\langle  Z_{1}-Z_{2},
      Z_{1}|Z_{1}| - Z_{2}|Z_{2}| \big\rangle
	 +
     3\zeta|x_{1} - x_{2}|^{2}
     \\&~+
     3\zeta|y_{1} - y_{2}|^{2}
     +
     3\zeta\Big||Z_{1}|^{\frac{3}{2}}
       - |Z_{2}|^{\frac{3}{2}})\Big|^{2}
\end{align*}
Noting that $|Z_{1}-Z_{2}|^{2} = |x_{1}-x_{2}|^{2} + (x_{1}-x_{2})(y_{1}-y_{2}) + \frac{1}{4}|y_{1}-y_{2}|^{2}$,
we take $\zeta = \frac{5}{4}$ and utilize \eqref{eq:55} as well as
\begin{align*}
      -\big\langle  Z_{1}-Z_{2},
      Z_{1}|Z_{1}| - Z_{2}|Z_{2}| \big\rangle
      \leq
      -\big(|Z_{1}| + |Z_{2}|\big)
      \big(|Z_{1}| - |Z_{2}|\big)^{2}
\end{align*}
to derive that for all $x_{1},x_{2},y_{1},y_{2} \in \R$,
\begin{align*}
      &~2\langle  (x_{1}-x_{2})-(D(y_{1})-D(y_{2})),
	  b(t,x_{1},y_{1})-b(t,x_{2},y_{2}) \rangle
	  +
	  \zeta
     |\sigma(t,x_{1},y_{1})
     -\sigma(t,x_{2},y_{2})|^{2}
     \\\leq&~
     -17|x_{1}-x_{2}|^{2}
     +
     \frac{26}{4}|y_{1}-y_{2}|^{2}
	 +
	 3\zeta|x_{1} - x_{2}|^{2}
     +
     3\zeta|y_{1} - y_{2}|^{2}
     \\&~
     -22\big(|Z_{1}| + |Z_{2}|\big)
     \big(|Z_{1}| - |Z_{2}|\big)^{2}
     +6\zeta(|Z_{1}| + |Z_{2}|)
     \big(|Z_{1}|-|Z_{2}|\big)^{2}
     \\\leq&~
     -13|x_{1}-x_{2}|^{2}
     +
     \frac{41}{4}|y_{1}-y_{2}|^{2}
     -
     \frac{1}{4}|x_{1}-x_{2}|^{2},
\end{align*}
i.e., $c_1=13,~c_2=10.25,~c_3=0.25,~c_4=0$.
Theorem \ref{th:bemsolution-stab} shows that the backward Euler method for \eqref{eq:examplestability} is mean square exponentially stable for any $\Delta < \log_{(\frac{c_1}{c_2})^{\frac{1}{\tau}}} \zeta \approx 3.75$, which has been numerically confirmed with six different stepsizes
in Figure \ref{fig:NSDDE-stability}.

Finally,
we consider the following two-dimensional nonlinear NSDDEs
\begin{align}\label{eq:examplestiff}
      &~\diff{(X(t)-D(X(t-\tau)))}
      +
      AX(t)\diff{t}\notag
	  \\=&~
	  b(X(t),X(t-\tau))\diff{t}
	  +
	  \sigma(X(t),X(t-\tau))\diff{W(t)},
	  \quad t \in (0,1]
\end{align}
with the initial value $X(t) = \phi(t) = (\cos t,\sin t)^{\top}$ for $t \in [-0.25,0]$, where the coefficient $D \colon \R^{2} \to \R^{2}$ and the symmetric positive definite matrix $A$ are
\begin{equation*}
      D(y)
      =
      \left[
      \begin{array}{c}
            -0.2(\sin y_{1})^{2} \\
            -0.2(\sin y_{2})^{2} \\
      \end{array}
      \right],
      \quad
      A=
      \frac{1}{2}
      \left[
      \begin{array}{cc}
            1+\alpha & 1-\alpha \\
            1-\alpha & 1+\alpha \\
      \end{array}
      \right]
\end{equation*}
for $y = (y_{1},y_{2})^{\top}$ and $\alpha \gg 0$. Besides, the functions $b \colon \R^{2} \times \R^{2} \to \R^{2}$ and $\sigma \colon \R^{2}\times\R^{2} \to \R^{2 \times 2}$ are given by
\begin{equation*}
      ~
      b(x,y)
      =
      \left[
      \begin{array}{c}
            x_{1} - 0.6\big(x_{1} + 0.2(\sin y_{1})^{2}\big)
            |x_{1} + 0.2(\sin y_{1})^{2}|
            \\
            x_{2} - 0.6\big(x_{2} + 0.2(\sin y_{2})^{2}\big)
            |x_{2} + 0.2(\sin y_{2})^{2}|
            \\
      \end{array}
      \right],
\end{equation*}
\begin{equation*}
      \sigma(x,y)
      =
      \beta
      \left[
      \begin{array}{cc}
            x_{1}+y_{1} & 0 \\
            0 & x_{2}+y_{2} \\
      \end{array}
      \right],
      \quad
      x
      =
      \left[
      \begin{array}{c}
            x_{1} \\
            x_{2} \\
      \end{array}
      \right],
      \quad
      y
      =
      \left[
      \begin{array}{c}
            y_{1} \\
            y_{2} \\
      \end{array}
      \right],
      \quad
      \beta \geq 0.
\end{equation*}
Taking $\alpha = 96$, $\beta = 0.04$ shows that the matrix $A$ admits a very large eigenvalue $\lambda = 96$, which leads to \eqref{eq:examplestiff} being a stiff system \cite{andersson2017meansquare}.

\begin{figure}[!htbp]
      \centering
      \includegraphics[width=0.5\textwidth]{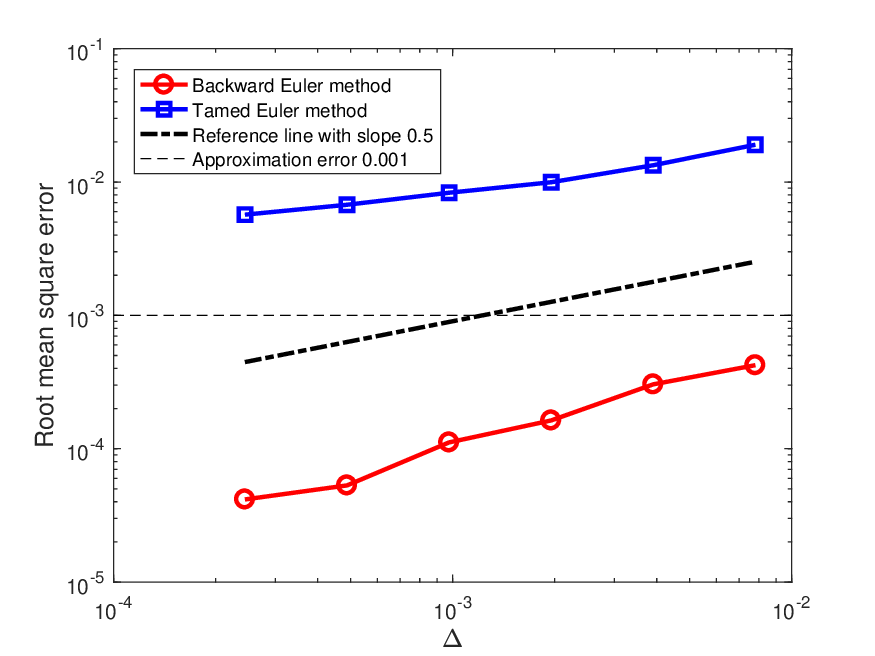}
      \caption{Convergence order simulations for \eqref{eq:examplestiff}}
      \label{fig:stiff}
\end{figure}

Recall that the tamed EM method proposed in \cite{deng2021tamed} applied to \eqref{eq:examplestiff} reads as $X_{n} = \phi(t_{n})$ for $n = -M, -M+1, \cdots, 0$ and
\begin{align}\label{eq:two_nonlinear_stiff_method}
      X_{n+1}
      =
      D(X_{n+1-M}) + X_{n} - D(X_{n-M})
      +
      \tfrac{(b(X_{n},X_{n-M})-AX_{n})\Delta}
      {1+\sqrt{\Delta}(|X_{n}|+|X_{n-M}|)}
      +
      \tfrac{\sigma(X_{n},X_{n-M})\Delta{W_n}}
      {1+\sqrt{\Delta}(|X_{n}|+|X_{n-M}|)}
\end{align}
for $n = 0, 1,\cdots,N$. Stepsize $\Delta = 2^{-16}$ and six different stepsizes $\Delta = 2^{-i}, i = 7,8,9,10,11,12$ are used to generated the unavailable exact solution and the other numerical approximations via \eqref{eq:BEmethodintro} and \eqref{eq:two_nonlinear_stiff_method}, respectively. Figure \ref{fig:stiff} shows that the backward Euler method achieves the precision $\epsilon = 0.001$ with the strong approximation problem $\big(\E[|X(T)-X_N|^{2}]\big)^{\frac{1}{2}} \leq \epsilon$ for all given stepsizes, but the tamed Euler method does not. Thus, the considered backward Euler method performs better than the explicit tamed EM method in terms of dealing with stiff NSDDEs.

\section{Conclusion}
This work investigates the convergence rate and exponential stability of the backward Euler method for NSDDEs as well as SDDEs under generalized monotonicity conditions. Our main results show that the investigated method converges strongly in the mean square sense with order $1/2$ and preserves the mean square exponential stability of the original equations. These theoretical results are finally validated by some numerical experiments. Concerning that strong convergence rates for numerical approximations are particular important to design efficient multilevel Monte Carlo approximation methods \cite{giles2008montecarlo}, there are two interesting ideas for our future work on NSDDEs satisfying generalized monotonicity conditions. One of them is to derive the strong convergence rate in $L^{p}(\Omega;\R^{d})$-norm with general $p \geq 2$ \cite{liu2022lpconvergence}, and the other is to develop numerical methods with higher order convergence rates \cite{wang2023meansquare}.

%
%
%
%
%
%
%
%
%
%
%
%
%
%
%
%
%

\bibliographystyle{abbrv}
\bibliography{NSDDERefs20240103}
\end{document}